\documentclass[11pt]{article}

\usepackage{times,latexsym,graphicx,amsmath,amsfonts,amsthm}
\usepackage{natbib}

\newtheoremstyle{localthm}
	{5pt} 
	{5pt} 
	{\sl} 
	{} 
	{\bf} 
	{{\rm.}} 
	{.7em} 
	{} 

\theoremstyle{localthm}
\newtheorem{Theorem}{Theorem}
\newtheorem{Proposition}{Proposition}

\newtheorem{Corollary}[Proposition]{Corollary}

\newtheoremstyle{localrem}
	{5pt} 
	{5pt} 
	{\rm} 
	{} 
	{\bf} 
	{{\rm.}} 
	{.7em} 
	{} 

\theoremstyle{localrem}
\newtheorem{Example}{Example}

\def\eps{\varepsilon}

\def\Ex{\operatorname{\mathbb{E}}}
\def\Pr{\operatorname{\mathbb{P}}}

\def\op{o_{\mathrm{p}}}
\def\Op{O_{\mathrm{p}}}
\def\top{\to_{\mathrm{p}}}

\def\Fnhat{\hat{F}_n}
\def\fnhat{\hat{f}_n}
\def\Pnhat{\hat{P}_n}
\def\phinhat{\hat{\varphi}_n}
\def\munhat{\hat{\mu}_n}
\def\Mnhat{\hat{M}_n}
\def\Wnhat{\hat{W}_n}

\def\Fnhatemp{\hat{F}_n^{\mathrm{emp}}}
\def\Pnhatemp{\hat{P}_n^{\mathrm{emp}}}
\def\munhatemp{\hat{\mu}_n^{\mathrm{emp}}}
\def\Mnhatemp{\hat{M}_n^{\mathrm{emp}}}
\def\Wnhatemp{\hat{W}_n^{\mathrm{emp}}}

\def\Snhat{\hat{\mathcal{S}}_n}

\def\R{\mathbb{R}}

\def\d{\mathrm{d}}

\def\AA{\mathcal{A}}
\def\DD{\mathcal{D}}

\begin{document}

\addtolength{\baselineskip}{+.2\baselineskip}

\title{On the tails of log-concave density estimators}

\author{Didier B.\ Ryter and Lutz D{\"u}mbgen\\
University of Bern}

\date{\today}

\maketitle

\paragraph{Abstract.}
It is shown that the nonparametric maximum likelihood estimator of a univariate log-concave probability density satisfies desirable consistency properties in the tail regions. Specifically, let $P$ and $f$ denote the true underlying distribution and density, respectively. If $\fnhat$ is the estimated log-concave density, and $\phinhat = \log \fnhat$, then we specify sequences $(b_n)_{n\in \mathbb{N}}$ such that $P([b_n,\infty)) \to 0$ at a specific speed, ensuring that the absolute errors or absolute relative errors of $\fnhat, \ \phinhat$ and $\phinhat'$ converge to zero uniformly on sets $[a, b_n]$. The main tools, besides characterizations of $\fnhat$, are exponential and maximal inequalities for truncated moments of log-concave distributions, which are of independent interest.

\paragraph{AMS subject classification:}
62G05, 62G07, 62G20, 62G32

\paragraph{Key words:}
Chernov bounds, consistency, exponential inequalities, log-linear densities.

\section{Introduction}
\label{sec:Introduction}

Let $f$ be a log-concave probability density on the real line, that is, $f(x) = \exp(\varphi(x))$ for some concave and upper semicontinuous function $\varphi : \R \to [-\infty,\infty)$. Suppose we observe independent random variables $X_1, \ldots, X_n$ with density $f$ and corresponding distribution function $F$. As noted by \cite{Walther_2002} and \cite{Pal_etal_2006}, for any sample size $n \ge 2$, there exists a unique maximum-likelihood estimator (MLE) $\fnhat = \exp(\phinhat)$ of $f$, where the MLE $\phinhat$ of $\varphi$ maximizes
\[
	\sum_{i=1}^n \psi(X_i) 
\]
over all concave functions $\psi : \R \to [-\infty,\infty)$ such that $\int e^{\psi(x)} \, \d x = 1$. Denoting the order statistics of $X_1,\ldots,X_n$ with $X_{(1)} < \cdots < X_{(n)}$, this estimator $\phinhat$ is piecewise linear on $[X_{(1)}, X_{(n)}]$ with changes of slope only at observations, and $\phinhat = -\infty$ outside of $[X_{(1)}, X_{(n)}]$.

Concerning consistency of $\phinhat$, let $\{x \in \R : 0 < F(x) < 1\} =: (a_o,b_o)$ with $-\infty \le a_o < b_o \le \infty$. It was shown by \cite{Duembgen_Rufibach_2009} that for any fixed interval $[a,b] \subset (a_o,b_o)$, the supremum of $|\phinhat - \varphi|$ over $[a,b]$ is of order $\Op(\rho_n^{1/3})$, where $\rho_n := \log(n)/n$. If $\varphi$ is H\"older-continuous on a neighborhood of $[a,b]$ with exponent $\beta \in (1,2]$, this rate improves to $\Op(\rho_n^{\beta/(2\beta+1)})$. Uniform consistency of $\phinhat$ on arbitrary compact subintervals of $(a_o,b_o)$ implies that $\int \bigl| \fnhat(x) - f(x) \bigr| \, \d x \top 0$. (Throughout this paper, asymptotic statements refer to $n \to \infty$.) Pointwise limiting distributions at a single point $x_o \in (a_o,b_o)$ have been derived by \cite{Balabdaoui_etal_2009}, assuming that $\varphi$ is twice continuously differentiable in a neighborhood of $x_o$. \cite{Kim_Samworth_2016} showed that the expected squared Hellinger distance between $\fnhat$ and $f$ is of order $O(n^{-4/5})$. Numerous further results about $\phinhat$ and $\fnhat$ have been derived thereafter, including multivariate settings, see the review of \cite{Samworth_2018}. In the present univariate setting, fast algorithms for the computation of $\phinhat$ and related objects are provided by \cite{Duembgen_Rufibach_2011} and \cite{Duembgen_etal_2021}. Experiments with simulated data show that even in the tail regions, that is, close to $a_o$ and $b_o$, the estimator $\phinhat$ is surprisingly accurate. In view of this empirical finding, \cite{Mueller_Rufibach_2009} developed new estimators for extreme value analysis with excellent empirical performance. However, the currently available theory about the asymptotic properties of $\phinhat$ does not explain its good performance in the tail regions.

In what follows, several results about $\phinhat(x)$ and the right-sided derivative $\phinhat'(x+)$ for $x$ close to $b_o$ are derived. By symmetry, these findings carry over to results in the left tail region. The main results are presented in Section~\ref{sec:Results} while Sections~\ref{sec:Auxiliary} and \ref{sec:Proofs} provide the proofs. Moreover, the results in Section~\ref{subsec:exp} are of independent interest.

\section{Main results}
\label{sec:Results}

We start with a simple consequence of the pointwise consistency of $\fnhat$, $\phinhat$ and concavity of $\varphi, \phinhat$.

\begin{Theorem}
\label{thm0}
The estimator $\fnhat$ does not overestimate $f$ in the sense that
\[
	\sup_{x \in \R} \, \bigl( \fnhat(x) - f(x) \bigr)^+ \ \top \ 0 .
\]
Moreover, for any sequence $(b_n)_n$ in $(a_o,b_o)$ with limit $b_o$,
\[
	\phinhat'(b_n+) \
	\begin{cases}
		\le \ \varphi'(b_o-) + \op(1)
			& \text{if} \ \varphi'(b_o-) > -\infty , \\
		\top \ -\infty
			& \text{if} \ \varphi'(b_o-) = -\infty ,
	\end{cases}
\]
where $\phinhat'(x+) := -\infty$ for $x \ge X_{(n)}$.
\end{Theorem}

The remaining goal is to show that the right tails are not ``severely underestimated'', and for this task we distinguish the cases $b_o = \infty$ and $b_o < \infty$.

\begin{Theorem}
\label{thm:finite.b0}
Suppose that $b_o < \infty$.

\noindent
{\bf (a)} \ Let $f(b_o) = 0$. Then for any fixed $a \in (a_o,b_o)$,
\[
	\sup_{x \ge a} \bigl| \fnhat(x) - f(x) \bigr| \ \top \ 0 .
\]
Moreover, for any given sequence $(b_n)_n$ in $(a_o,b_o)$ with limit $b_o$,
\[
	\phinhat'(b_n+) \ \top \ \varphi'(b_o-) = -\infty .
\]

\noindent
{\bf (b)} \ Let $f(b_o) > 0$. Then for arbitrary fixed intervals $[a,b_n] \subset (a_o,b_o)$ such that $b_n \uparrow b_o$ and $n (1 - F(b_n)) \to \infty$,
\[
	\sup_{x \in [a,b_n]} \bigl| \phinhat(x) - \varphi(x) \bigr| \ \top \ 0 .
\]

\noindent
{\bf (c)} \ Let $f(b_o) > 0$ and $\varphi'(b_o-) > -\infty$. Then for any given sequence $(b_n)_n$ in $(a_o,b_o)$ such that $b_n \uparrow b_o$ and $\rho_n^{-1/3} (1 - F(b_n)) \to \infty$,
\[
	\phinhat'(b_n+) \ \top \ \varphi'(b_o-) .
\]
\end{Theorem}

\begin{Example}
We illustrate Theorem~\ref{thm0} and Theorem~\ref{thm:finite.b0}~(b-c) for samples from the uniform distribution on $[0,1]$. Figure~\ref{fig:finite.b0sample} depicts the functions $\phinhat$ (left panel) and $\phinhat'(\cdot +)$ (right panel) for one ``typical sample'' of size $n = 150$ (top), $n = 500$ (middle) and $n=2000$ (bottom). Figure~\ref{fig:finite.b0} shows the performance of $\phinhat$ in $10000$ simulations of a sample of size $n = 150$ (top), $n = 500$ (middle) and $n=2000$ (bottom). The left panels show the estimated $\gamma$-quantiles of $\phinhat(x)$, $x \in (0,1)$, for $\gamma = 0.01, 0.1, 0.25, 0.5, 0.75, 0.9, 0.99$. For the same values of $\gamma$, the right panels show the estimated $\gamma$-quantiles of $\phinhat'(x+)$, $x \in (0,1)$. As expected, the estimators $\phinhat$ and $\phinhat'(\cdot+)$ suffer from a substantial bias very close to the boundaries $0$ and $1$, but these problematic regions shrink as the sample size $n$ increases.
\end{Example}

\begin{figure}
\includegraphics[width=0.5 \textwidth]{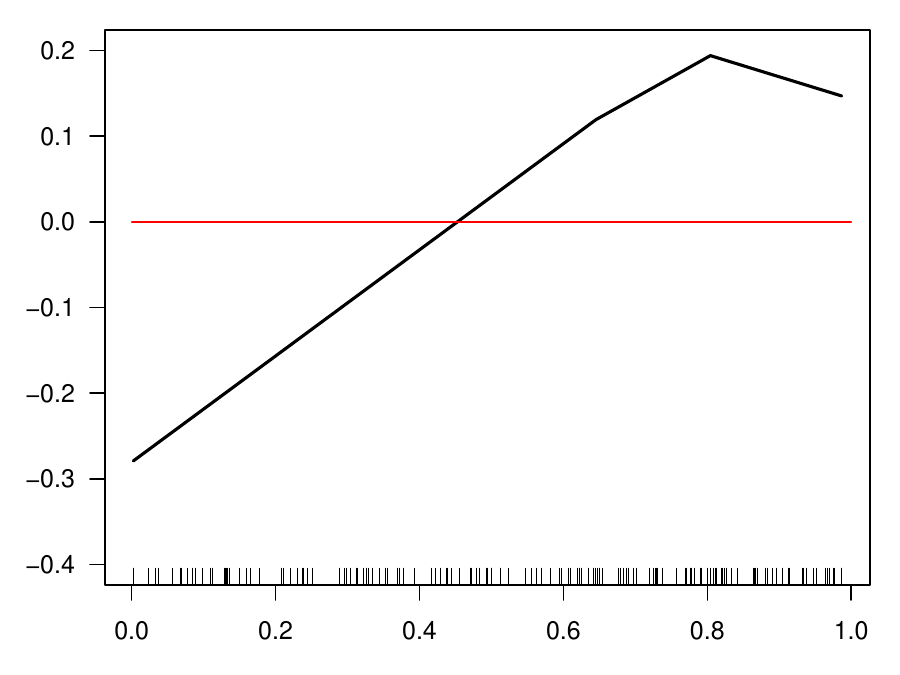}%
\hfill
\includegraphics[width=0.5 \textwidth]{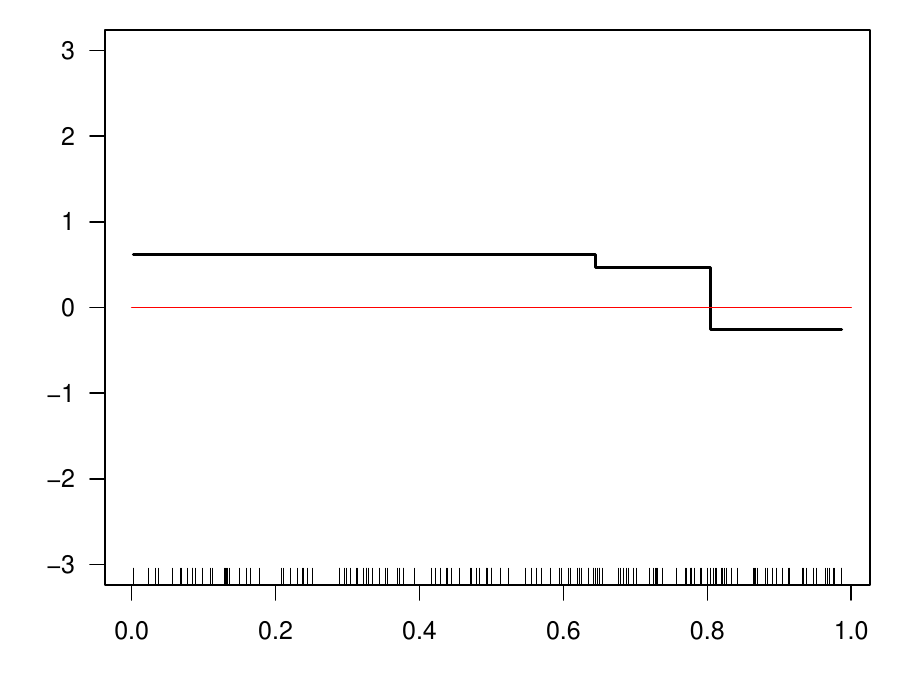}
\vfill
\includegraphics[width=0.5 \textwidth]{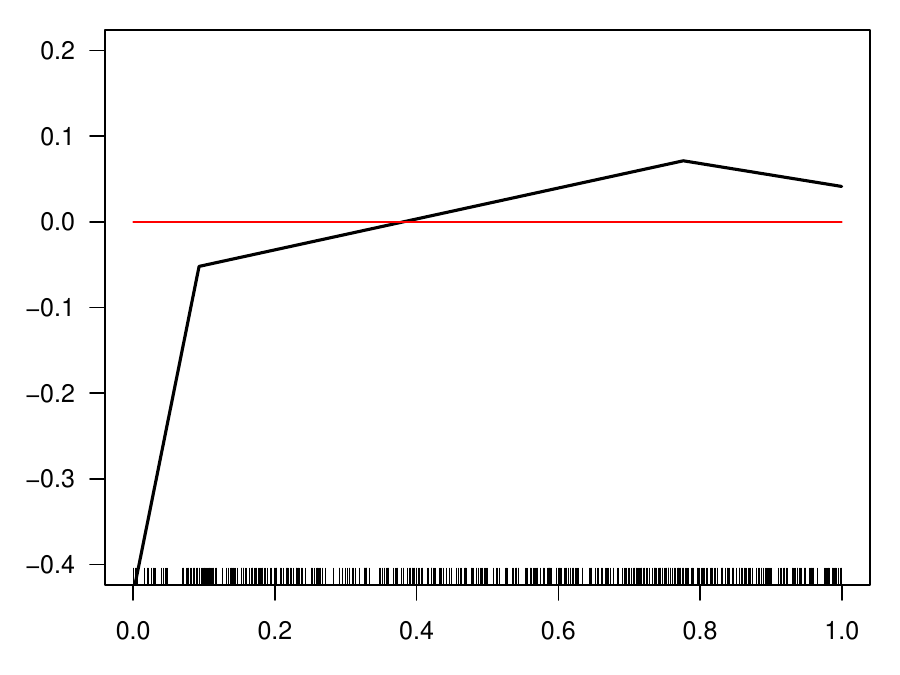}%
\hfill
\includegraphics[width=0.5 \textwidth]{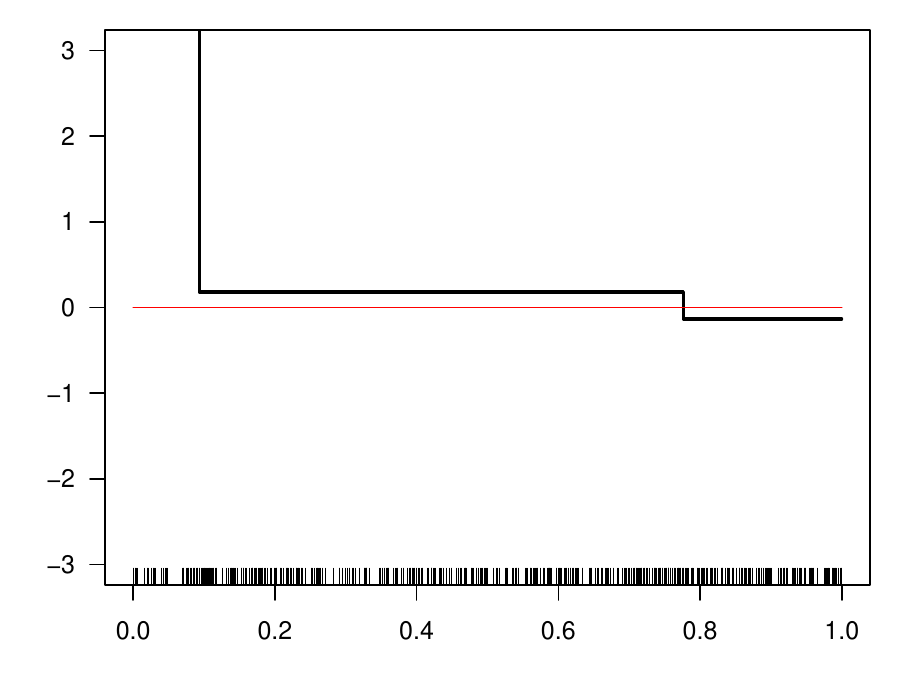}
\vfill
\includegraphics[width=0.5 \textwidth]{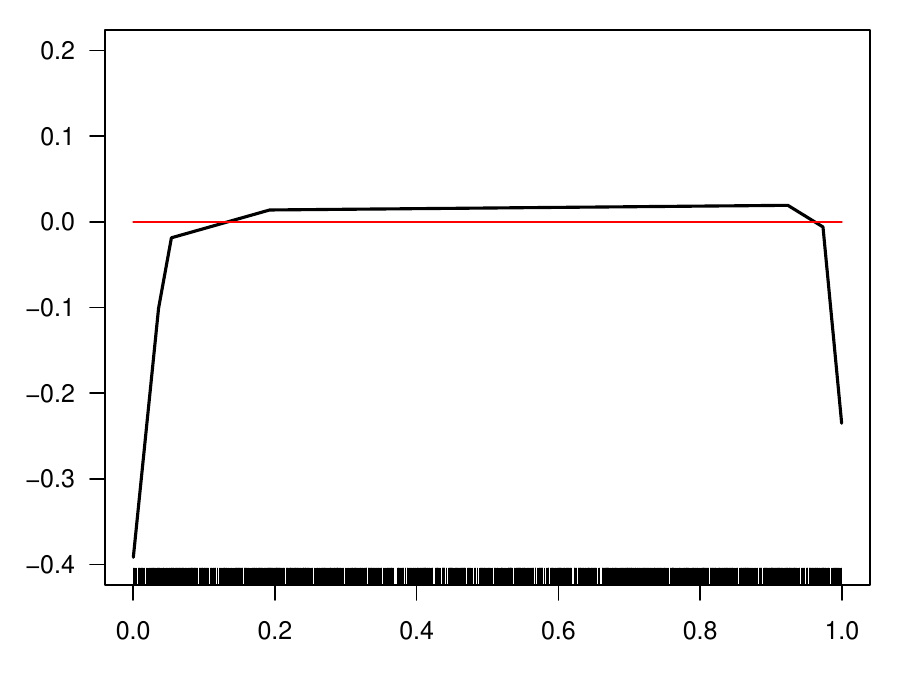}%
\hfill
\includegraphics[width=0.5 \textwidth]{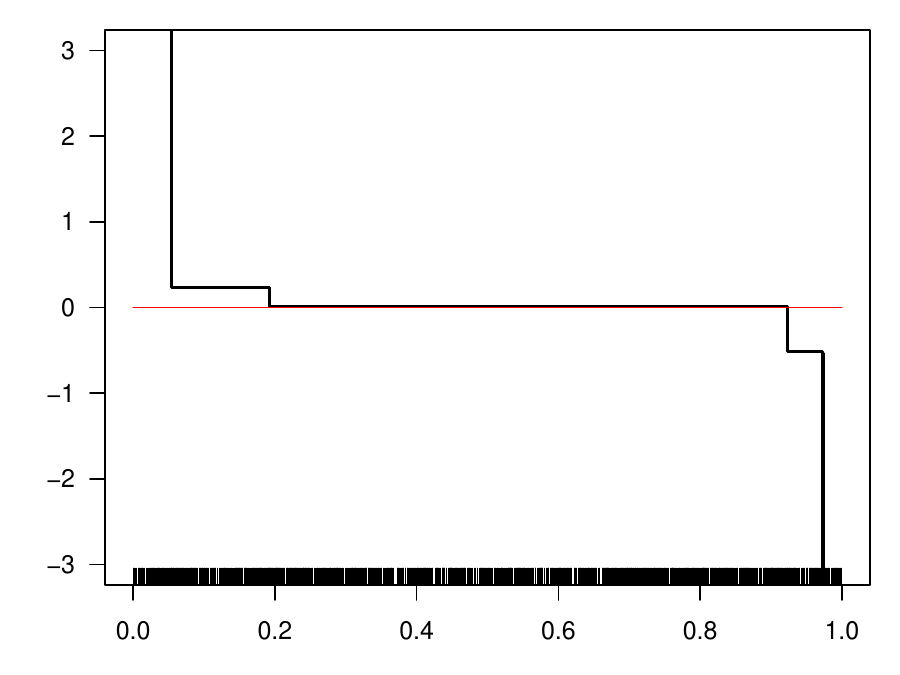}
\caption{The functions $\phinhat$ (left panel) and $\phinhat'(\cdot+)$ (right panel) for one particular sample of size $n = 150$ (top), $n = 500$ (middle) and $n=2000$ (bottom) from $\mathrm{Unif}[0,1]$. The sample is indicated as a rug plot, and the true values $\varphi$ and $\varphi'$ are shown in red.}
\label{fig:finite.b0sample}
\end{figure}

\begin{figure}
\includegraphics[width=0.5 \textwidth]{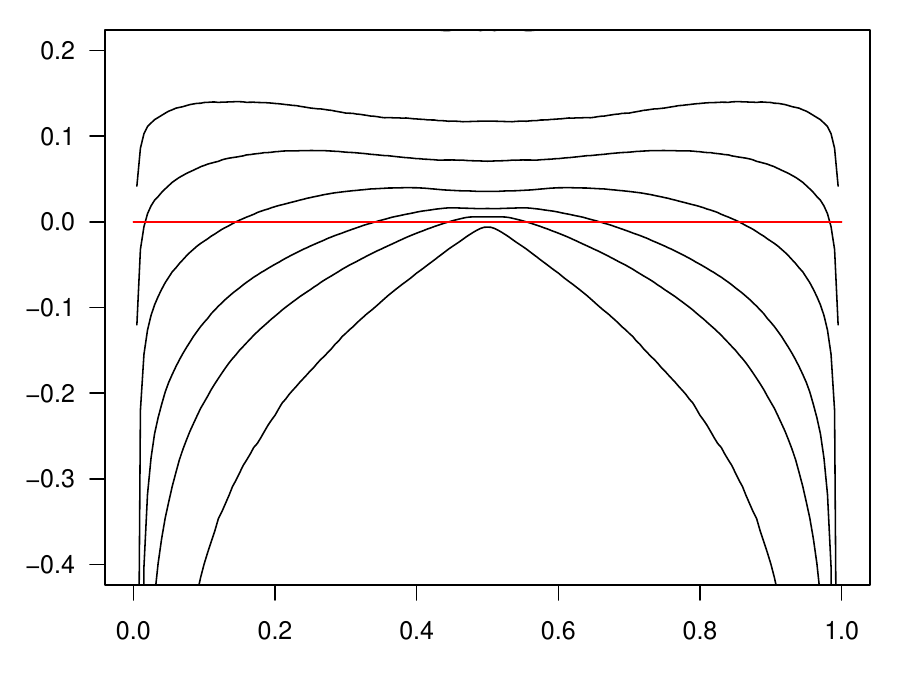}%
\hfill
\includegraphics[width=0.5 \textwidth]{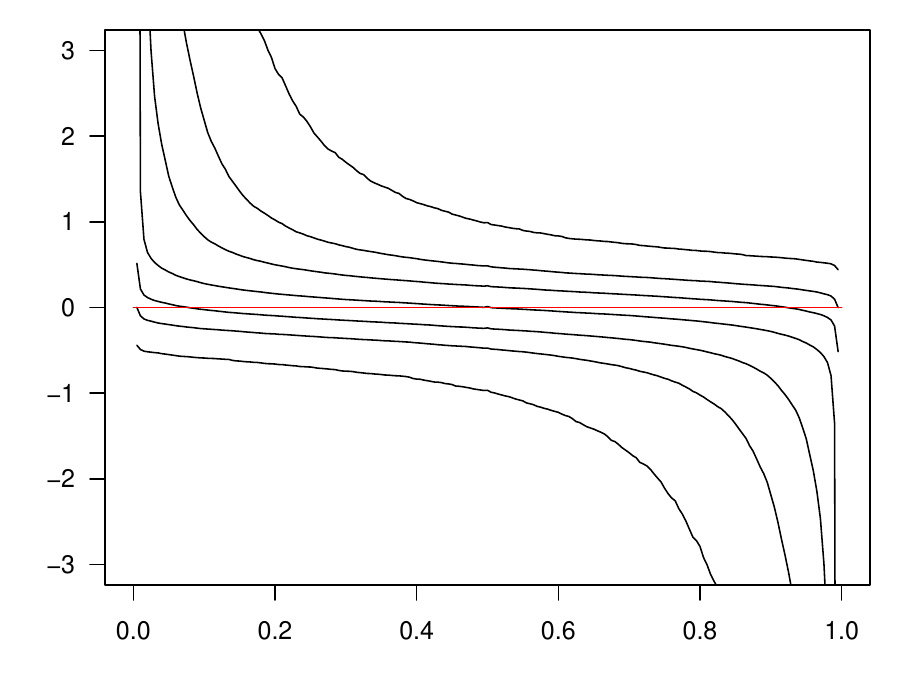}
\vfill
\includegraphics[width=0.5 \textwidth]{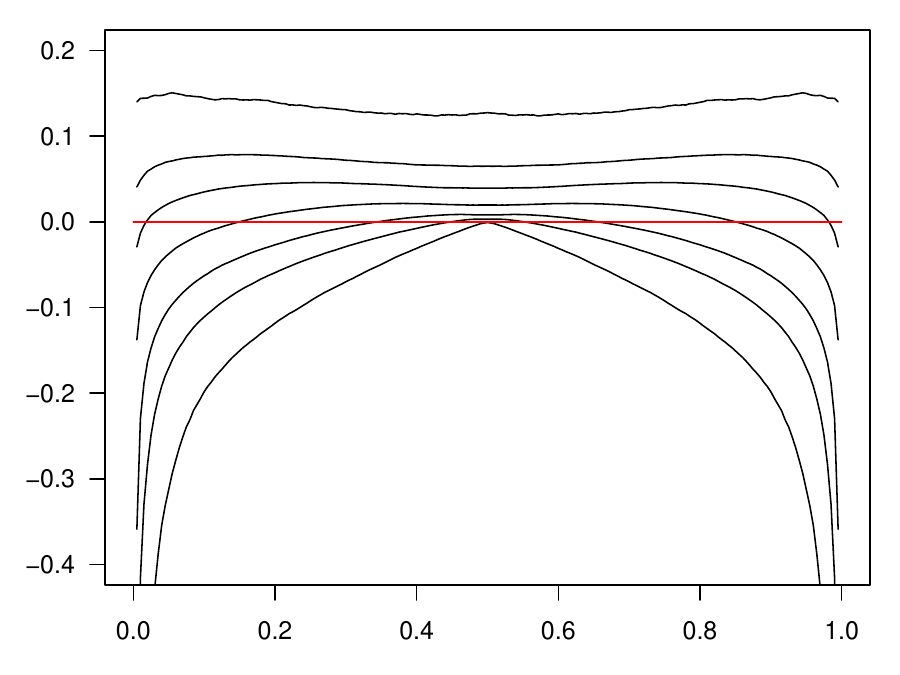}%
\hfill
\includegraphics[width=0.5 \textwidth]{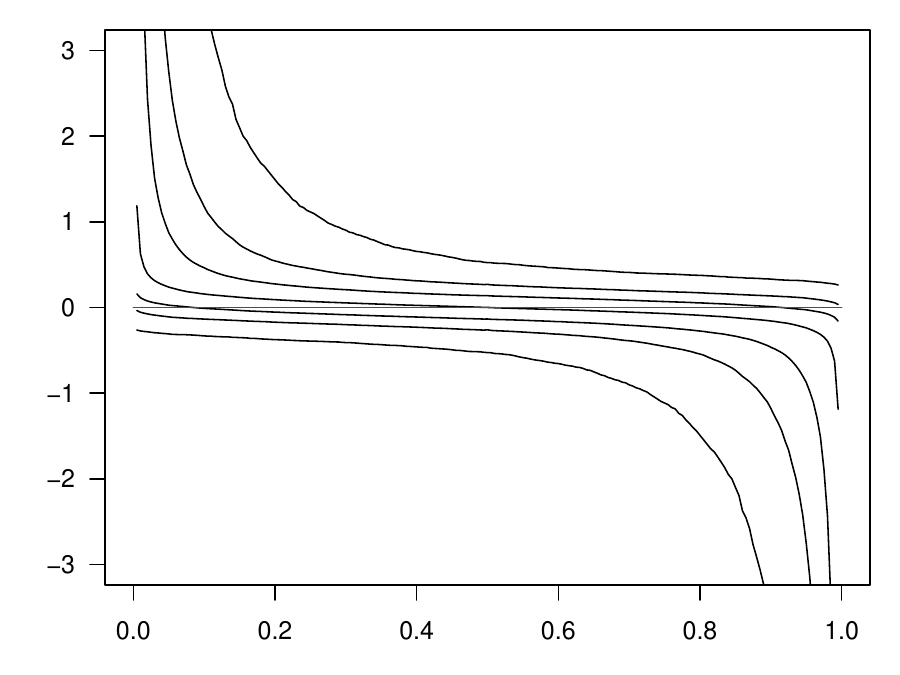}
\vfill
\includegraphics[width=0.5 \textwidth]{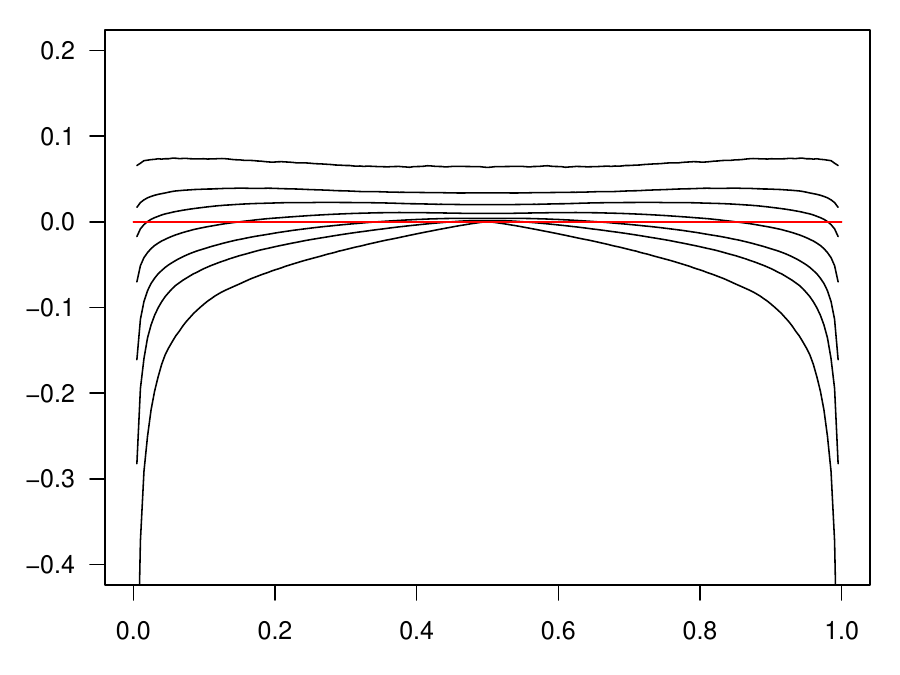}%
\hfill
\includegraphics[width=0.5 \textwidth]{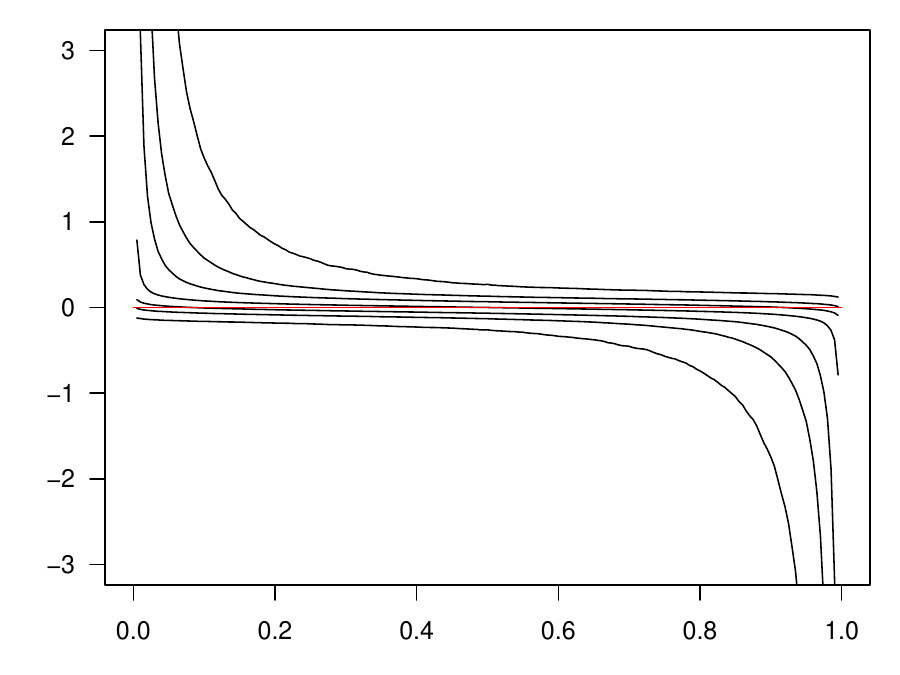}
\caption{Estimated $\gamma$-quantiles of $\phinhat(x)$ (left panel) and $\phinhat'(x+)$ (right panel) for samples of size $n = 150$ (top), $n = 500$ (middle) and $n=2000$ (bottom) from $\mathrm{Unif}[0,1]$. The true values $\varphi(x)$ and $\varphi'(x)$ are shown in red.}
\label{fig:finite.b0}
\end{figure}

If the support of $P$ is unbounded to the right, then $\varphi(x) \to -\infty$ and $\varphi'(x+) \to \varphi'(\infty-) \in [-\infty, 0)$ as $x \to \infty$. Here are some results complementing Theorem~\ref{thm0}.

\begin{Theorem}
\label{thm:infinite.b0}
Suppose that $b_o = \infty$. Let $(b_n)_n$ be a sequence in $(a_o,\infty)$ such that $b_n \to \infty$ and $(1 - F(b_n))/\rho_n \to \infty$.

\noindent
{\bf (a)} \ With asymptotic probability one, $\fnhat(b_n) > 0$, and
\[
	\phinhat'(b_n+) \ \top \ \varphi'(\infty-) .
\]

\noindent
{\bf (b)} \ Suppose that $\varphi'(\infty-) > -\infty$. Then for any fixed $a \in (a_o,\infty)$,
\[
	\max_{x \in [a, b_n]} \, \frac{ \bigl| \phinhat(x) - \varphi(x) \bigr|}{1 + |\varphi(x)|} \
	\top \ 0 .
\]

\noindent
{\bf (c)} \ Suppose that $\varphi'(\infty-) = -\infty$ and $\varphi$ is differentiable on some halfline $(a_*,\infty) \subset (a_o,\infty)$ with Lipschitz-continuous derivative $\varphi'$. Then, for arbitrary fixed $a \in (a_*,\infty)$ such that $\varphi'(a) < 0$,
\[
	\sup_{x \in [a,b_n]} \, \Bigl( \frac{\phinhat'(x+)}{\varphi'(x)} - 1 \Bigl)^+ \
	\top \ 0
\]
and
\[
	\sup_{x \in [a,b_n]} \, \bigl( \phinhat(x) - \varphi(x) \bigr)^+ \
	\top \ 0 ,
	\quad
	\sup_{x \in [a,b_n]} \, \frac{\bigl( \varphi(x) - \phinhat(x) \bigr)^+}{1 + |\varphi(x)|} \
	\top \ 0 .
\]
\end{Theorem}

Parts~(a) and (b) apply, for instance, if $P$ is a logistic distribution or a gamma distribution with shape parameter in $[1,\infty)$. Parts~(a) and (c) explain why the estimator $\phinhat$ is remarkably accurate in the tails if, for instance, $P$ is a Gaussian distribution. In particular, for any fixed $\eps > 0$,
\[
	\Pr \bigl( \fnhat(x) \le (1 + \eps) f(x) \
		\text{for all} \ x \in [a,b_n] \bigr) \
	\to \ 1 ,
\]
which is substantially stronger than the first conclusion of Theorem~\ref{thm0}. Furthermore, since $\varphi' < 0$ on $[a,b_n]$,
\[
	\Pr \bigl( \phinhat'(x+) \ge (1 + \eps) \varphi'(x) \
		\text{for all} \ x \in [a,b_n] \bigr) \
	\to \ 1 ,
\]
which complements the second conclusion of Theorem~\ref{thm0}.

\begin{Example}
We illustrate Theorem~\ref{thm0} and Theorem~\ref{thm:infinite.b0}~(b) for samples from the standard Gaussian distribution.
Figure~\ref{fig:infinite.b0sample} depicts the functions $\phinhat$ (left panel) and $\phinhat'(\cdot +)$ (right panel) for one ``typical sample'' of size $n=150$ (top), $n=500$ (middle) and $n=2000$ (bottom). Figure~\ref{fig:infinite.b0} shows the performance of $\phinhat$ in $10000$ simulations of a sample of size $n = 150$ (top), $n=500$ (middle) and $n=2000$ (bottom). The left-hand side shows the estimated $\gamma$-quantiles of $\phinhat(x)$, $x \in (-4,4)$ for $\gamma = 0.01, 0.1, 0.25, 0.5, 0.75, 0.9, 0.99$.  For the same values of $\gamma$, the right-hand side shows the estimated $\gamma$-quantiles of $\phinhat'(x+)$, $x \in (-4,4)$.
\end{Example}

\begin{figure}
\includegraphics[width=0.5 \textwidth]{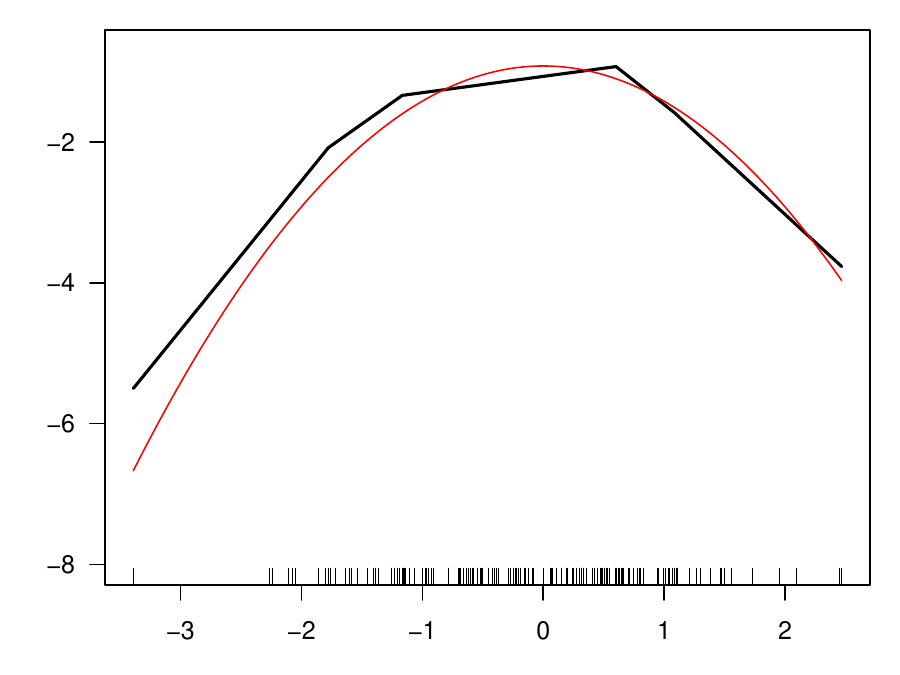}%
\hfill
\includegraphics[width=0.5 \textwidth]{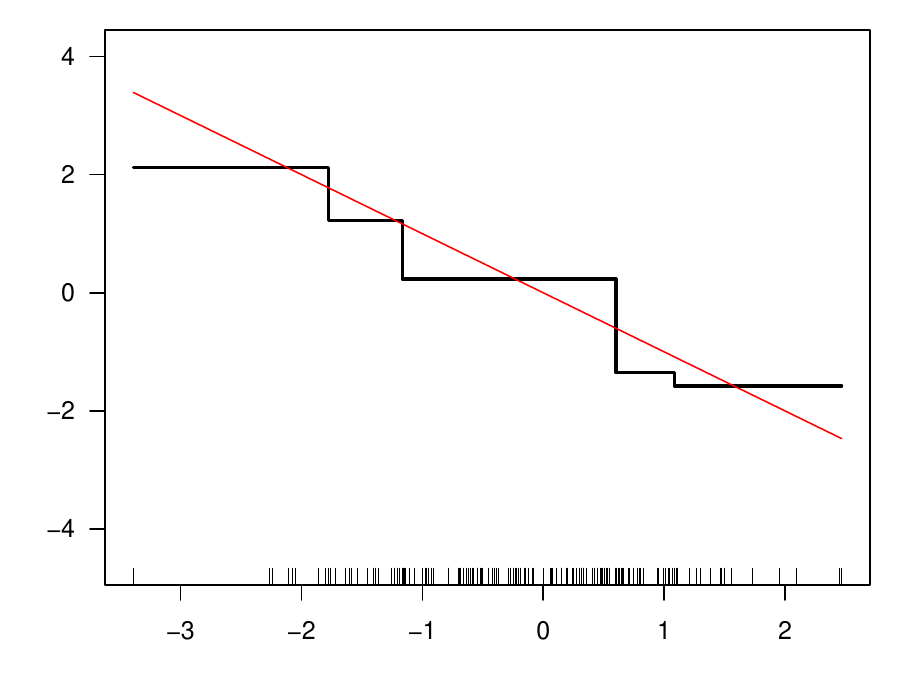}
\vfill
\includegraphics[width=0.5 \textwidth]{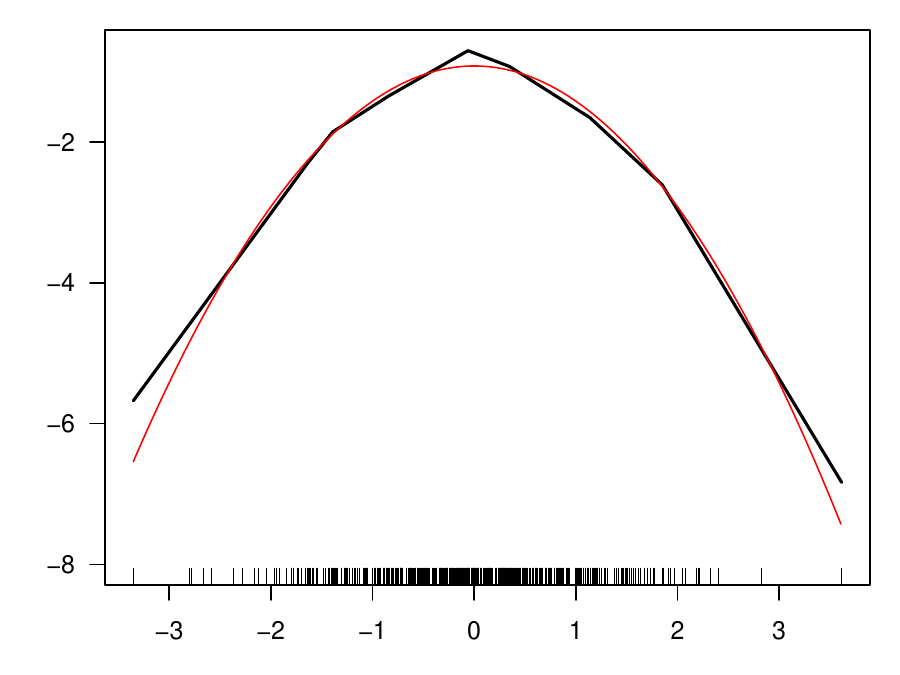}%
\hfill
\includegraphics[width=0.5 \textwidth]{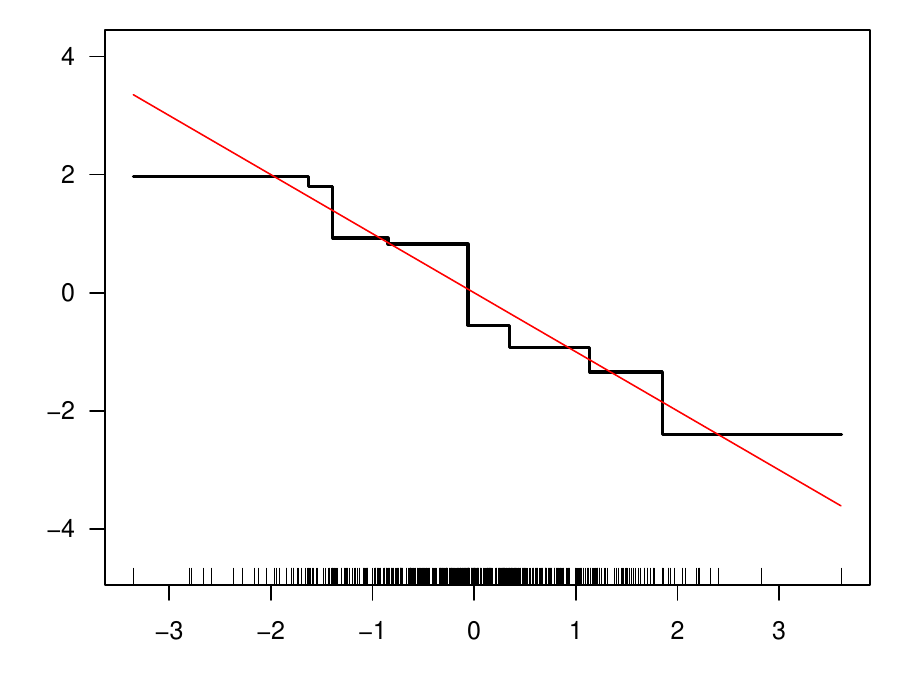}
\vfill
\includegraphics[width=0.5 \textwidth]{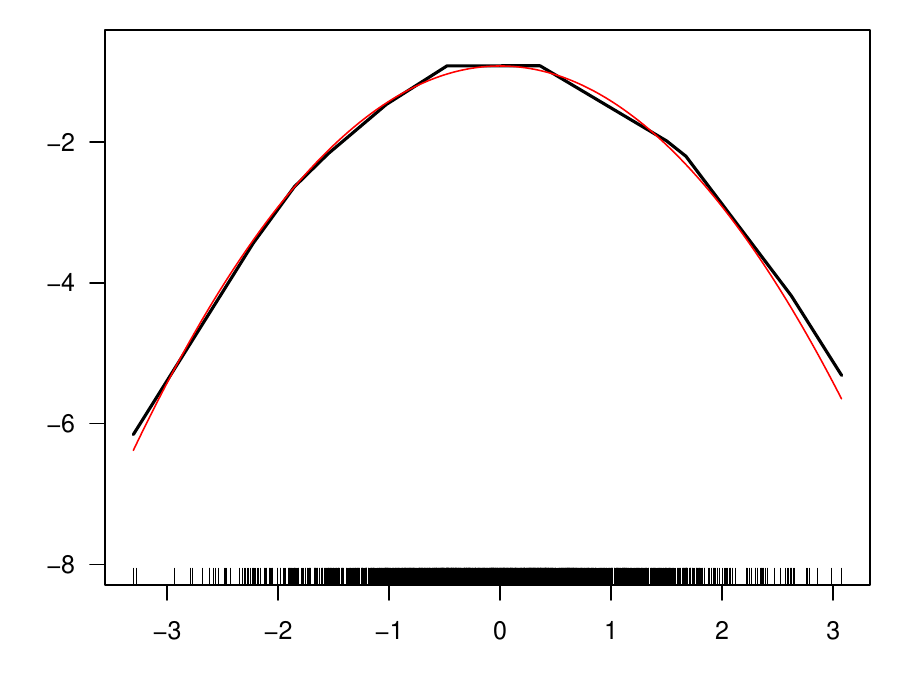}%
\hfill
\includegraphics[width=0.5 \textwidth]{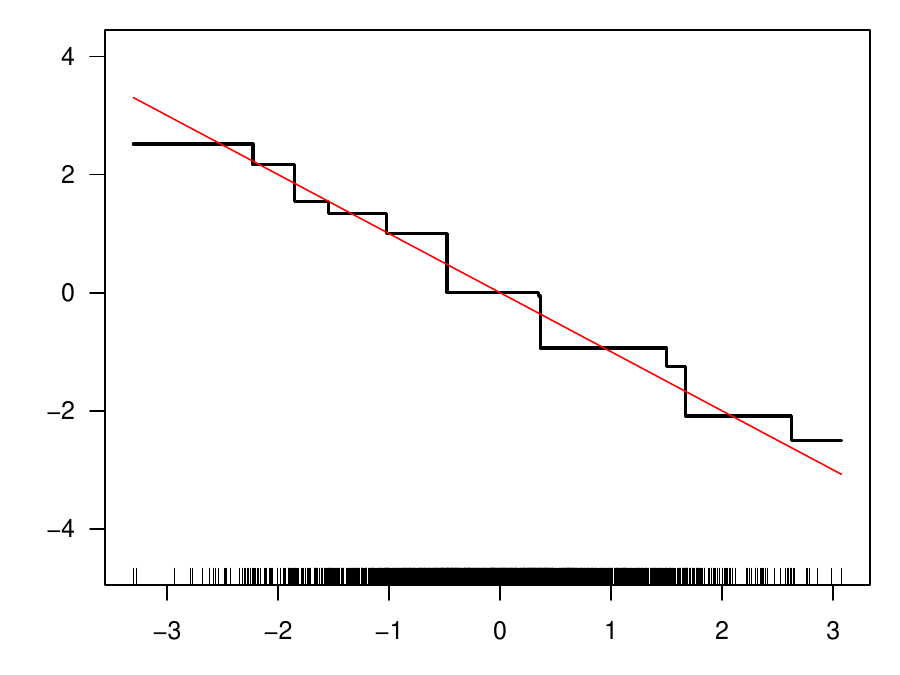}
\caption{The functions $\phinhat$ (left panel) and $\phinhat'(\cdot+)$ (right panel) for one particular sample of size $n = 150$ (top), $n = 500$ (middle) and $n=2000$ (bottom) from $\mathrm{N}(0,1)$. The sample is indicated as a rug plot, and the true values $\varphi$ and $\varphi'$ are shown in red.}
\label{fig:infinite.b0sample}
\end{figure}

\begin{figure}
\includegraphics[width=0.5 \textwidth]{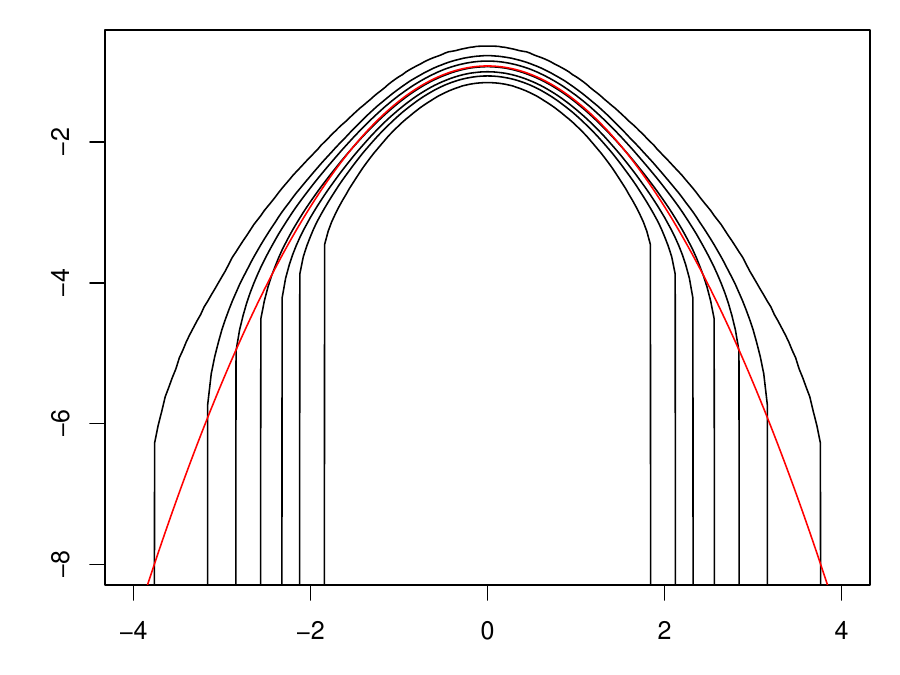}%
\hfill
\includegraphics[width=0.5 \textwidth]{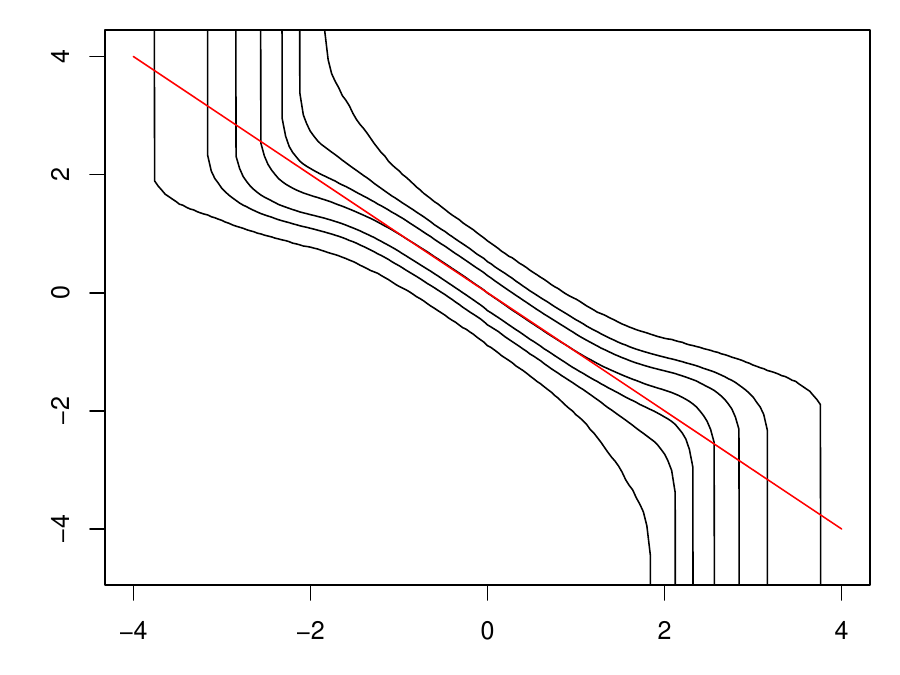}
\vfill
\includegraphics[width=0.5 \textwidth]{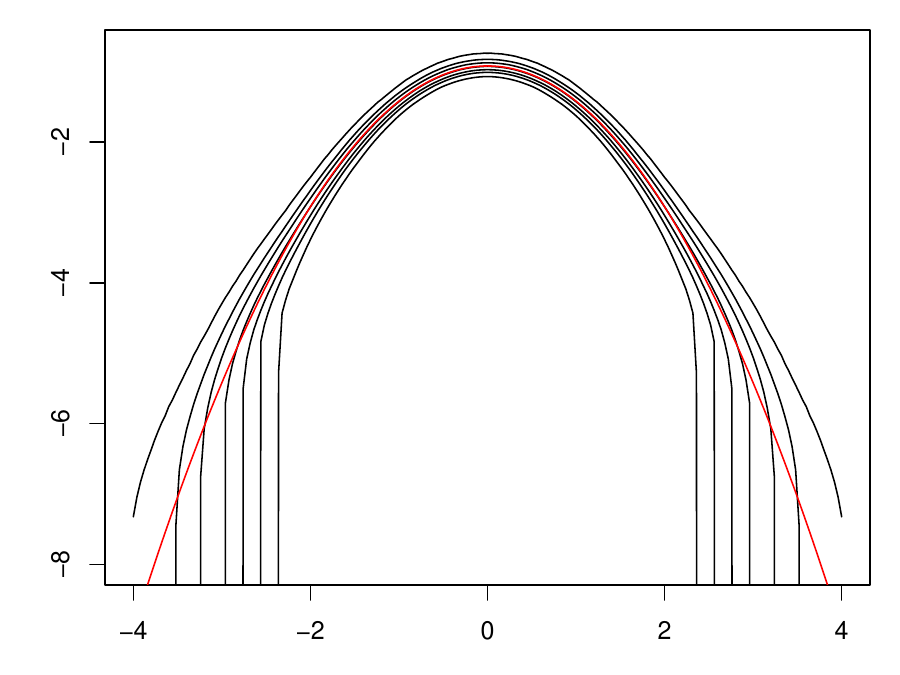}%
\hfill
\includegraphics[width=0.5 \textwidth]{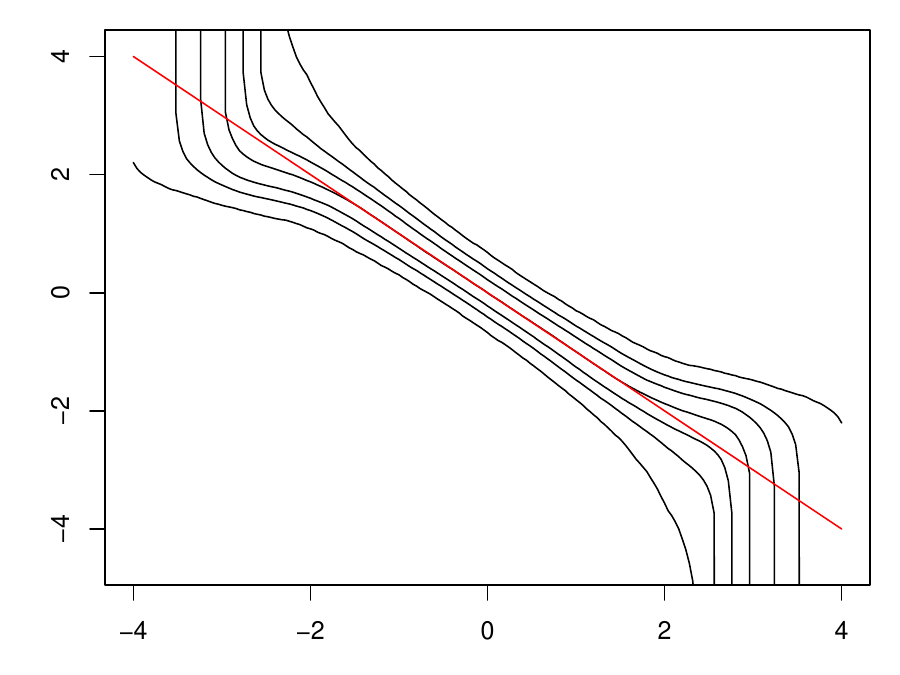}
\vfill
\includegraphics[width=0.5 \textwidth]{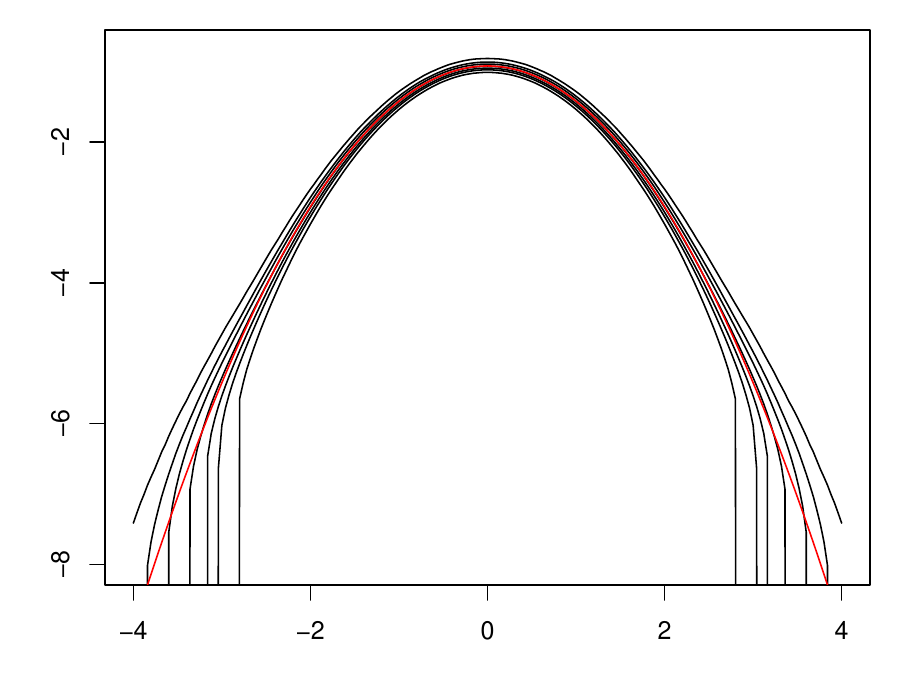}%
\hfill
\includegraphics[width=0.5 \textwidth]{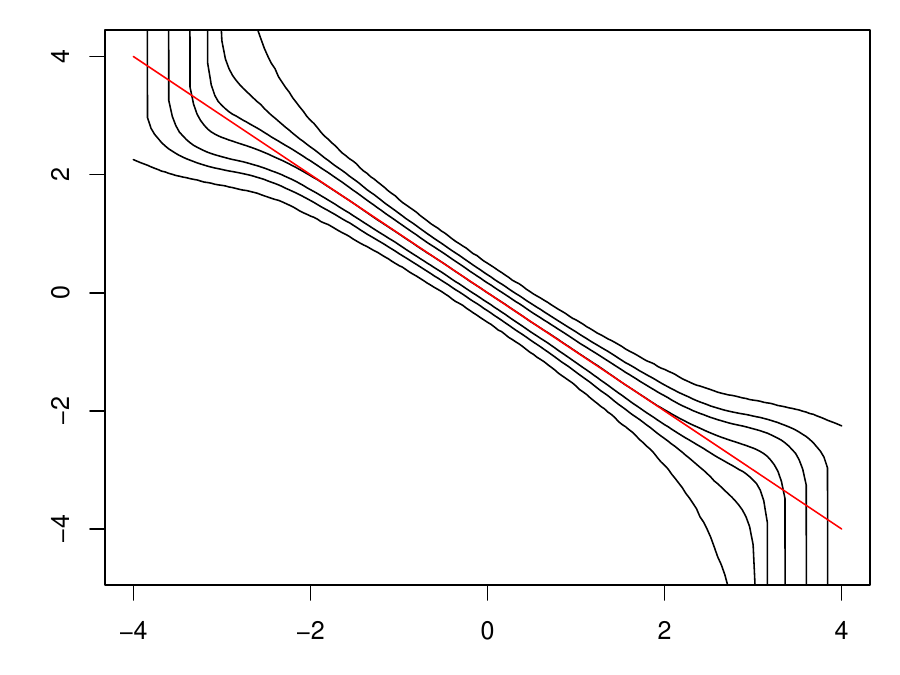}
\caption{Estimated $\gamma$-quantiles of $\phinhat(x)$ (left panel) and $\phinhat'(x+)$ (right panel) for samples of size $n = 150$ (top), $n = 500$ (middle) and $n=2000$ (bottom) from $\mathrm{N}(0,1)$. The true values $\varphi(x)$ and $\varphi'(x)$ are shown in red.}
\label{fig:infinite.b0}
\end{figure}

\section{Auxiliary results}
\label{sec:Auxiliary}

In what follows, let $P$ and $\Pnhat$ be the distribution with density $f$ and $\fnhat$, respectively. The corresponding distribution functions are denoted by $F$ and $\Fnhat$, respectively. In addition, let $\Pnhatemp$ and $\Fnhatemp$ be the empirical distribution and the empirical distribution function, respectively, of the observations $X_1,\ldots,X_n$.

\subsection{More about $\fnhat$ and $\phinhat$}

We mentioned already some properties of $\phinhat$ and $\fnhat$. Recall that for any fixed $[a,b] \subset (a_o,b_o)$,
\begin{equation}
\label{eq:consistency.phi.ab}
	\sup_{x \in [a,b]} \, \bigl| \phinhat(x) - \varphi(x) \bigr| \
	\top \ 0 ,
\end{equation}
and since $\varphi$ is bounded on $[a,b]$, this implies that
\begin{equation}
\label{eq:consistency.f.ab}
	\sup_{x \in [a,b]} \, \bigl| \fnhat(x) - f(x) \bigr| \
	\top \ 0 .
\end{equation}
An important consequence of the latter result is that
\begin{equation}
\label{eq:consistencyTV}
	\sup_{B \in \mathrm{Borel}(\R)} \bigl| \Pnhat(B) - P(B) \bigr| \
	= \ 2_{}^{-1} \int \bigl| \fnhat(x) - f(x) \bigr| \, \d x \
	\top \ 0 ,
\end{equation}
see \cite{Duembgen_Rufibach_2009}. The latter paper also provides the following key inequalities: Let
\[
	\Snhat \ := \ \{X_{(1)},X_{(n)}\}
		\cup \bigl\{ x \in (X_{(1)}, X_{(n)}) : \phinhat(x-) > \phinhat(x+) \bigr\} ,
\]
the set of kinks of $\phinhat$. Then for arbitrary $b \in \R$,
\begin{equation}
\label{eq:char.1}
	\int (x - b)^+ \, \Pnhatemp(\d x) \ 
	\begin{cases}
		\displaystyle
		\ge \ \int (x - b)^+ \, \Pnhat(\d x) , \\
		\displaystyle
		= \ \int (x - b)^+ \, \Pnhat(\d x)
			& \text{if} \ b \in \Snhat .
	\end{cases}
\end{equation}
Moreover, for $b \in \Snhat$,
\begin{equation}
\label{eq:char.2}
	\Fnhatemp(b-) \ \le \ \Fnhat(b) \ \le \ \Fnhatemp(b)
\end{equation}
Finally,
\[
	\int x \, \Pnhatemp(\d x) \ = \ \int x \, \Pnhat(\d x) .
\]

\subsection{Inequalities for $\Pnhatemp$}

Concerning $\Fnhatemp$, note the following useful inequality: For any $b < b_o$, 
\[
	\Ex \biggl( \sup_{x \le b} \, \Bigl| \frac{\Fnhatemp(x) - F(x)}{1 - F(x)} \Bigr|^2 \biggr) \
	\le \ \frac{4}{n(1 - F(b))} .
\]
This follows from the well-known fact that $M_x := [\Fnhatemp(x) - F(x)]/[1 - F(x)]$ defines a martingale $(M_x)_{x < b_o}$ and from one of Doob's martingale inequalities \citep{Shorack_Wellner_1986,Hall_Heyde_1980}. In particular, for any sequence of numbers $b_n \in (a_o,b_o)$,
\begin{equation}
\label{eq:tails.Fnhatemp}
	\sup_{x \le b_n} \Bigl| \frac{1 - \Fnhatemp(x)}{1 - F(x)} - 1 \Bigr| \
	\top \ 0
	\quad \text{if} \ n(1 - F(b_n)) \to \infty .
\end{equation}
Combining this with \eqref{eq:char.2} leads to the fact that
\begin{equation}
\label{eq:tails.Fnhat}
	\max_{x \in \Snhat : x \le b_n} \Bigl| \frac{1 - \Fnhat(x)}{1 - F(x)} - 1 \Bigr| \
	\top \ 0
	\quad \text{if} \ n(1 - F(b_n)) \to \infty .	
\end{equation}

Here is another useful result about $\Pnhatemp(I) - P(I)$ over real intervals $I$.

\begin{Proposition}[Consistency of $\Pnhatemp$]
\label{prop:Pnhatemp}
For any constant $\tau > 2$, with asymptotic probability one,
\[
	\bigl| \Pnhatemp(I) - P(I) \bigr| \
	\le \ \sqrt{2 \tau \rho_n \min\{\Pnhatemp(I),P(I)\}} + (\tau + 2) \rho_n
\]
for arbitrary intervals $I \subset \R$.
\end{Proposition}

\begin{proof}[Proof of Proposition~\ref{prop:Pnhatemp}]
Note first that for arbitrary indices $0 \le j < k \le n+1$ with $k - j \le n$, the random variable $P(X_{(j)}, X_{(k)})$ follows a beta distribution with parameters $k-j$ and $n+1-k+j$, see Chapter~3.1 of \cite{Shorack_Wellner_1986}. Here we set $X_{(0)} := a_o$ and $X_{(n+1)} := b_o$. In particular, its mean equals
\[
	p_{njk} \ := \ \frac{k-j}{n+1} ,
\]
and it follows from Proposition~2.1 of \cite{Duembgen_1998} that for any $\eta > 0$,
\[
    \Pr \bigl[ \Psi \bigl( P(X_{(j)},X_{(k)}), p_{njk} \bigr) \ge \eta \rho_n \bigr] \
	\le \ 2 \exp( - (n+1) \eta \rho_n) \ < \ 2 n^{-\eta} ,
\]
where $\Psi(x,p) := p \log(p/x) + (1 - p) \log[(1 - p)/(1 - x)]$. Moreover, for $c > 0$, the inequality $\Psi(x,p) < c$ implies that $|x - p| < \sqrt{2cp(1 - p)} + |1 - 2p| c$. Furthermore, since $\Psi(x,p) = K_0(x,p) = K_1(p,x)$ in the notation of \cite{Duembgen_Wellner_2023}, Lemma~S.12 in the latter paper shows that $|x - p| < \sqrt{2cx} + (2/3) |1 - 2x| c$ whenever $\Psi(x,p) < c$. Consequently, the probability that
\begin{equation}
\label{ineq:Pnhatemp}
	\bigl| P(X_{(j)},X_{(k)}) - p_{njk} \bigr| \
	\le \ \sqrt{2 \eta \rho_n \min\{P(X_{(j)},X_{(k)}), p_{njk}\}} + \eta \rho_n
\end{equation}
for all indices $0 \le j < k \le n+1$ is at least $1 - 2 \binom{n+2}{2} n^{-\eta} = 1 - (1 + 3/n + 2/n^2) n^{2-\eta}$, and this converges to $1$ if $\eta > 2$.

It remains to be shown that in case of $\eta \in (2,\tau)$, inequality \eqref{ineq:Pnhatemp} implies that
with asymptotic probability one,
\[
	\bigl| P(I) - \Pnhatemp(I) \bigr| \
	\le \ \sqrt{2 \tau \rho_n \min\{P(I),\Pnhatemp(I)\}} + (\tau + 2) \rho_n
\]
for all intervals $I \subset \R$. Note first that
\begin{equation}
\label{ineq:Pnhatemp2}
	\max_{\ell=1,\ldots,n+1} P(X_{(\ell-1)}, X_{(\ell)}) \
	\le \ \rho_n + \Op(n^{-1}) .
\end{equation}
This can be deduced from the well-known representation $P(X_{(\ell-1)},X_{(\ell)}) = E_\ell / \sum_{i=1}^{n+1} E_i$ with independent, standard exponential random variables $E_1, \ldots, E_{n+1}$. Now, for an arbitrary nonvoid interval $I$, let $j = j_n(I) \in \{0,1,\ldots,n\}$ be maximal and $k = k_n(I) \in \{1,2,\ldots,n+1\}$ be minimal such that $I \subset [X_{(j)},X_{(k)}]$. If $k - j \le 2$, then
\[
	\Pnhatemp(I) - P(I) \ \le \ \Pnhatemp(I) \ \le \ 3/n \ = \ o(\rho_n)
\]
and
\[
	P(I) - \Pnhatemp(I) \ \le \ P(I) \ \le \ 2 \max_{\ell=1,\ldots,n+1} P(X_{(\ell-1)}, X_{(\ell)})
	\ \le \ (2 + \op(1)) \rho_n .
\]
In case of $k - j \ge 3$, let $\tilde{I} := [X_{(j+1)},X_{(k-1)}] \subset I$ with
\[
	\Pnhatemp(\tilde{I}) = \frac{k-j-1}{n} \ > \ \frac{k-j-2}{n+1} = p_{n,j+1,k-1} .
\]
Consequently, it follows from \eqref{ineq:Pnhatemp} and \eqref{ineq:Pnhatemp2} that
\begin{align*}
	P(I) - \Pnhatemp(I) \
	&\le \ P(\tilde{I}) - \Pnhatemp(\tilde{I})
		+ 2 \max_{\ell=0,\ldots,n} P(X_{(\ell-1)}, X_{(\ell)}) \\
	&\le \ P(\tilde{I}) - p_{n,j+1,k-1} + (2 + \op(1)) \rho_n \\
	&\le \ \sqrt{2 \eta \rho_n \min\{P(\tilde{I}),p_{n,j+1,k-1}\}} + (\eta + 2 + \op(1)) \rho_n \\
	&\le \ \sqrt{2 \eta \rho_n \min\{P(I),\Pnhatemp(I)\}} + (\eta + 2 + \op(1)) \rho_n ,
\end{align*}
and
\begin{align*}
	\Pnhatemp(I) - P(I)
	&\le \ 3/n + (1 + 1/n) p_{n,j+1,k-1} - P(\tilde{I}) \\
	&\le \ 4/n + (1 + 1/n) \bigl( p_{n,j+1,k-1} - P(\tilde{I}) \bigr) \\
	&\le \ o(\rho_n) + (1 + 1/n)
		\Bigl( \sqrt{2 \eta \rho_n \min\{P(I),\Pnhatemp(I)\}} + \eta \rho_n \Bigr) \\
	&= \ \sqrt{2 (\eta + o(1)) \rho_n \min\{P(I),\Pnhatemp(I)\}} + (\eta + o(1)) \rho_n ,
\end{align*}
where the terms $\op(1)$ and $o(1)$ depend only on $n$, not on the interval $I$.
\end{proof}

\subsection{Truncated and conditional means of $P$}

For $-\infty < a < b \le \infty$, let
\begin{align*}
	M(a,b) \
	&:= \ \int_{(a,b)} (x - a) \, P(\d x) , \\
	\mu(a,b) \
	&:= \ \begin{cases}
			P(a,b)^{-1} M(a,b) & \text{if} \ P(a,b) > 0 , \\
			0 & \text{else} .
		\end{cases}
\end{align*}
For $-\infty \le a < b < \infty$, we set
\[
	W(a,b) \ := \ \int_{(a,b)} (b - x) \, P(\d x) .
\]
Further, let
\begin{equation}
\label{eq:def.mua}
	\mu(a) \ := \ \mu(a,\infty) .
\end{equation}
The univariate function $\mu$ is known as the mean excess function or mean residual lifetime in fields such as extreme value theory and actuarial science.

To formulate various approximations and inequalities for these functions $M$, $\mu$ and $W$, we need three auxiliary functions and some properties thereof.

\begin{Proposition}
\label{prop:N.and.nu}
Let $N, \nu, V : \R \to (0,\infty)$ be given by
\[
	N(t) \ := \ \int_0^1 u e^{tu} \, \d u , \quad
	\nu(t) \ := \ N(t) \Big/ \int_0^1 e^{tu} \, \d u , \quad
	V(t) \ := \ \int_0^1 (1 - u) e^{tu} \, \d u .
\]
These functions are continuously differentiable with $N', \nu', V' > 0$, where $N(0) = \nu(0) = V(0) = 1/2$, $N'(0) = 1/3$, $\nu'(0) = 1/12$ and $V'(0) = 1/6$. Moreover,
\[
	\lim_{t \to -\infty} \, t^2 N(t) \
	= \ \lim_{t \to -\infty} \, |t| \nu(t) \
	= \ \lim_{t \to -\infty} \, |t| V(t) \ = \ 1 .
\]
\end{Proposition}

This proposition follows from elementary calculus. The limits of $t^2 N(t)$, $|t| \nu(t)$ and $|t| V(t)$ as $t \to -\infty$ follow from the explicit formulae
\[
	N(t) \ = \ \frac{1 - (1-t)e^t}{t^2} , \quad
	\nu(t) \ = \ \frac{1 - (1 - t)e^t}{t(e^t - 1)} , \quad
	V(t) \ = \ \frac{-1 + (e^t - 1)/t}{t}
\]
for $t \ne 0$.

The next proposition summarizes several useful properties of the functions $M$, $\mu$ and $W$ and their relation to $f$ and $\varphi$. The monotonicity property of $\mu$ in part~(a) was noted already by \cite{Bagnoli_Bergstrom_2005}. In the proposition's proof and later on, we use repeatedly well-known results about the stochastic and likelihood ratio orders between probability distributions on the real line, see \cite{Shaked_Shanthikumar_2007} for the foundations. Specifically, let $P_1$ and $P_2$ be probability distributions on the real line with densities $f_1$ and $f_2$, respectively. If $f_1 \ge f_2$ on $(-\infty,x_o)$ and $f_1 \le f_2$ on $(x_o,\infty)$ for some real number $x_o$, then $P_1 \le_{\rm st} P_2$, where $\le_{\rm st}$ denotes stochastic order. In particular, if $f_2/f_1$ is non-decreasing on $\{f_1 + f_2 > 0\}$, then such a number $x_o$ has to exist, whence $P_1 \le_{\rm st} P_2$.

\begin{Proposition}[Properties of $M$ and $\mu$]\strut
\label{prop:M.and.mu}

\noindent
{\bf (a)} \ The function $\mu$ given by \eqref{eq:def.mua} is non-increasing and Lipschitz-continuous with constant one.

\noindent
{\bf (b)} \ Let $-\infty < a < b < c \le \infty$ such that $P(a,b) > 0$. Then,
\[
	\frac{1}{p} \ \le \ \frac{M(a,c)}{M(a,b)} \ \le \ \frac{1}{p + (1 - p)\log(1 - p)} ,
\]
where $p := P(a,b)/P(a,c) \in (0,1]$ and $0 \log 0 := 0$.

\noindent
{\bf (c)} \ For arbitrary real numbers $a < b$ in $[a_o, b_o]$,
\[
	f(a) (b - a)^2 N \bigl( \varphi'(b-)(b - a) \bigr) \
	\le \ M(a,b) \
	\le \ f(a) (b - a)^2 N \bigl( \varphi'(a+)(b - a) \bigr) ,
\]
\[
	f(a) (b - a)^2 V \bigl( \varphi'(b-)(b - a) \bigr) \
	\le \ W(a,b) \
	\le \ f(a) (b - a)^2 V \bigl( \varphi'(a+)(b - a) \bigr) ,
\]
where $N(-\infty), V(-\infty) := 0$ and $N(\infty), V(\infty) := \infty$. Moreover,
\[
	(b - a) \nu \bigl( \varphi'(b-)(b - a) \bigr) \
	\le \ \mu(a,b) \
	\le \ (b - a) \nu \bigl( \varphi'(a+)(b - a) \bigr)
\]
where $\nu(-\infty) := 0$, $\nu(\infty) := 1$.

\noindent
{\bf (d)} \ Suppose that $b_o = \infty$. Then for arbitrary real $a \in [a_o,\infty)$ with $\varphi'(a+) < 0$,
\[
	\frac{f(a)}{\varphi'(\infty-)^2} \
	\le \ M(a,\infty) \
	\le \ \frac{f(a)}{\varphi'(a+)^2}
	\quad\text{and}\quad
	\frac{1}{|\varphi'(\infty-)|} \
	\le \ \mu(a) \
	\le \ \frac{1}{|\varphi'(a+)|} .
\]

\noindent
{\bf (e)} \ Suppose that $\varphi$ is differentiable on some interval $(a,b) \subset [a_o,b_o]$ with $a \in \R$ and $\varphi'(a+) < 0$. Further, let $\varphi'$ be Lipschitz-continuous with constant $L$ on $(a,b)$. Then
\[
	M(a,b) \
	\ge \ \exp \bigl( - 3 L \varphi'(a+)^{-2} \bigr) \cdot
		\begin{cases}
			\displaystyle
			\frac{f(a)}{\varphi'(a+)^2}
				& \text{if} \ b = \infty , \\[2ex]
			f(a) (b - a)^2 N \bigl( \varphi'(a+)(b - a) \bigr)
				& \text{if} \ b < \infty ,
	\end{cases}
\]
and
\[
	\mu(a,b) \
	\ge \ \exp \bigl( - 3 L \varphi'(a+)^{-2} \bigr) \cdot
	\begin{cases}
		\displaystyle
		\frac{1}{|\varphi'(a+)|}
			& \text{if} \ b = \infty , \\[2ex]
		(b - a) \nu \bigl( \varphi'(a+)(b - a) \bigr)
			& \text{if} \ b < \infty .
	\end{cases}
\]
\end{Proposition}

\begin{proof}[Proof of Proposition~\ref{prop:M.and.mu}]
As to part~(a), by means of Fubini's theorem we may write
\[
	\mu(a) \
	= \ (1 - F(a))^{-1} \int \int_0^\infty 1_{[r < y - a]} \, \d r \, P(\d y) \
	= \ \int_0^\infty \frac{1 - F(a+r)}{1 - F(a)} \, \d r
\]
for $a \in \R$ with $F(a) < 1$. For $a < a'$ with $F(a') = 1$,
\[
	\mu(a) - \mu(a') \
	= \ \mu(a) \
	= \ \int_0^{a'-a} \frac{1 - F(a+r)}{1 - F(a)} \, \d r \
	\in \ [0, a'-a] .
\]
For $a < a'$ with $F(a') < 1$,
\begin{align*}
	\mu(a) - \mu(a') \
	&= \ \int_0^\infty \frac{1 - F(a+r)}{1 - F(a)} \, \d r
		- \int_0^\infty \frac{1 - F(a'+r)}{1 - F(a')} \, \d r \\
	&\le \ \int_0^\infty \frac{1 - F(a+r)}{1 - F(a)} \, \d r
		- \int_0^\infty \frac{1 - F(a'+r)}{1 - F(a)} \, \d r \\
	&= \ \int_a^{a'} \frac{1 - F(a+r)}{1 - F(a)} \, \d r \
		\le \ a' - a .
\end{align*}
That $\mu(a) \ge \mu(a')$ was noted already by \cite{Bagnoli_Bergstrom_2005}, but for the reader's convenience, we provide an argument here. For $\xi \in \{a,a'\}$, one may write
\[
	\mu(\xi) \ = \ \int_{[0,\infty)} z f_\xi(z) \, \d z
\]
with the probability density $f_\xi$ on $[0,\infty)$ given by $f_\xi(z) \ := \ \exp \bigl( \varphi(\xi + z) - \log(1 - F(\xi)) \bigr)$. By concavity of $\varphi$, $f_a/f_{a'}$ is non-decreasing on $\{f_a > 0\} = [0,b_o-a) \supset \{f_{a'} > 0\}$. This implies that the distribution with density $f_a$ is stochastically greater than (or equal to) the distribution with density $f_{a'}$. In particular, the mean $\mu(a)$ of the former is not smaller than the mean $\mu(a')$ of the latter.

To prove part~(b), it suffices to consider the nontrivial case that $P(b,c) > 0$, i.e.\ $p \in (0,1)$. Note that the ratios $M(a,c)/M(a,b)$, $\mu(a,c)/\mu(a,b)$ and $p = P(a,b)/P(a,c)$ remain the same if we replace $P$ with the conditional distribution $P(a,c)^{-1} P(\cdot \cap (a,c))$. Thus we may assume that $P(a,c) = 1$, and we may replace $c$ with $\infty$. With a random variable $X$ with this (modified) distribution $P$, we may write
\[
	\frac{M(a,c)}{M(a,b)} \ = \ \frac{\Ex(X - a)}{\Ex(1_{[X \le b]} (X - a))} ,
	\quad
	\frac{\mu(a,c)}{\mu(a,b)} \ = \ \frac{\Ex(X - a)}{\Ex(X - a \,|\, X \le b)} \
		= \ \frac{p M(a,c)}{M(a,b)} .
\]
Now let $\tilde{f}(x) := 1_{[x > a]} \lambda \exp( - \lambda (x - a))$ with $\lambda > 0$ such that $\int_a^b \tilde{f}(x) \, \d x = p$, that is,
\[
	\lambda \ = \ \frac{- \log(1 - p)}{b - a} .
\]
By concavity of $\log f$ and linearity of $\log \tilde{f}$ on $(a,\infty)$, either $f \equiv \tilde{f}$, or there exist numbers $a \le x_1 < x_2 \le \infty$ such that
\[
	f \ \begin{cases}
		< \tilde{f} \ \text{on} \ (a,x_1) \cup (x_2,\infty) , \\
		> \tilde{f} \ \text{on} \ (x_1,x_2) .
	\end{cases}
\]
Note that $x_1 = a$ would imply that $f > \tilde{f}$ on $(a,b)$ or $f < \tilde{f}$ on $(b,\infty)$, and in both cases we would end up with $\int_a^b f(x) \, \d x \ne \int_a^b \tilde{f}(x) \, \d x$. Similarly one can exclude the cases $x_2 = \infty$, $x_1 \ge b$ and $x_2 \le b$. Consequently, we know that there exist constants $a < x_1 < b < x_2 < \infty$ such that
\[
	f \ \begin{cases}
		\le \ \tilde{f} \ \ \text{on} \ (a,x_1) \cup (x_2,\infty) , \\
		\ge \ \tilde{f} \ \ \text{on} \ (x_1,x_2) .
	\end{cases}
\]
In particular, if $\tilde{X}$ is a random variable with density $\tilde{f}$, then
\[
	\mathcal{L}(X \,|\, X \le b) \ \ge_{\rm st}^{} \ \mathcal{L}( \tilde{X} \,|\, \tilde{X} \le b)
	\quad\text{and}\quad
	\mathcal{L}(X \,|\, X > b) \ \le_{\rm st}^{} \ \mathcal{L}( \tilde{X} \,|\, \tilde{X} > b) ,
\]
where $\mathcal{L}(\cdot)$ stands for `distribution'. In particular,
\begin{align*}
	\frac{M(a,c)}{M(a,b)} \
	&= \ 1 + \frac{(1 - p) \Ex(X -a \,|\, X > b)}
		{p \Ex(X - a \,|\, X \le b)} \\
	&\le \ 1 + \frac{(1 - p) \Ex(\tilde{X} -a \,|\, \tilde{X} > b)}
		{p \Ex(\tilde{X} - a \,|\, \tilde{X} \le b)} \
	= \ \frac{\Ex(\tilde{X} - a)}{\Ex(1_{[\tilde{X} \le b]} (\tilde{X} - a))} .
\end{align*}
Elementary calculations reveal that the latter ratio is equal to $1 / [p + (1 - p)\log(1 - p)]$, which yields our upper bound for $M(a,c)/M(a,b)$. Concerning the lower bound, note that
\[
	\frac{M(a,c)}{M(a,b)} \
	= \ 1 + \frac{(1 - p) \Ex(X -a \,|\, X > b)}
		{p \Ex(X - a \,|\, X \le b)} \
	\ge \ 1 + \frac{(1 - p)}{p} \
	= \ \frac{1}{p} .
\]

To prove part~(c), we consider up to three functions $\psi_1, \psi_2, \psi_3 : [0,b-a] \to [-\infty,\infty]$ given by
\begin{align*}
	\psi_1(z) \
		&:= \ \varphi'(b-) z \qquad (\text{if} \ \varphi'(b_o-) > -\infty) , \\
	\psi_2(z) \
		&:= \ \varphi(a + z) - \varphi(a) , \\
	\psi_3(z) \
		&:= \ \varphi'(a+) z \qquad (\text{if} \ \varphi'(a+) < \infty) .
\end{align*}
Concavity of $\varphi$ implies that $\psi_1 \le \psi_2 \le \psi_3$, so
\[
	\frac{M(a,b)}{f(a)} \
	= \ \int_{(a,b)} (x - a) \exp(\psi_2(x - a)) \, \d x \
	= \ \int_0^{b-a} z \exp(\psi_2(z)) \, \d z
\]
lies between
\[
	\int_0^{b-a} z \exp(\psi_1(z)) \, \d z \
	= \ (b-a)^2 N \bigl( \varphi'(b-)(b-a) \bigr)
\]
and
\[
	\int_0^{b-a} z \exp(\psi_3(z)) \, \d z \
	= \ (b-a)^2 N \bigl( \varphi'(a+)(b-a) \bigr) .
\]
Analogously, one can show that $W(a,b)/f(a)$ lies between $(b-a)^2 V \bigl( \varphi'(b-)(b - a) \bigr)$ and $(b-a)^2 V \bigl( \varphi'(a+)(b-a) \bigr)$.

Concerning the inequalities for $\mu(a,b)$, note that in case of $\varphi'(b-) > -\infty$,
\[
	\psi_2(z) - \psi_1(z) \ = \ \int_0^z \bigl[ \varphi'((a + t)+) - \varphi'(b-) \bigr] \, \d t
\]
is non-decreasing in $z \in [0,b-a]$, so the probability distribution $P_1$ on $[0,b-a]$ with density proportional to $\exp(\psi_1)$ is stochastically smaller than (or equal to) the probability distribution $P_2$ on $[0, b-a]$ with density proportional to $\exp(\psi_2)$. Thus,
\[
	\mu(a,b) \ = \ \int_{[0,b-a]} z \, P_2(\d z) \
	\ge \ \int_{[0,b-a]} z \, P_1(\d z) \
	= \ (b - a) \nu \bigl( \varphi'(b-)(b-a) \bigr) .
\]
Analogously, if $\varphi'(a+) < \infty$, then $\psi_3 - \psi_2$ is non-decreasing on $[0,b-a]$, and this implies that
\[
	\mu(a,b) \ \le \ (b - a) \nu \bigl( \varphi'(a+)(b-a) \bigr) .
\]

Part~(d) is verified similarly as part~(c). Here we consider two or three probability densities $f_1, f_2, f_3$ on $[0,\infty)$ given by $f_1(z) := \lambda_1 \exp( - \lambda_1 z)$ with $\lambda_1 := - \varphi'(\infty-)$ (if $\varphi'(\infty-) > - \infty$), $f_2(z) := \exp \bigl( \varphi(a + z) - \log(1 - F(a)) \bigr)$ and $f_3(z) := \lambda_3 \exp(- \lambda_3 z)$ with $\lambda_3 := - \varphi'(a+)$. If $\varphi'(\infty-) > -\infty$, then $f_2/f_1$ is non-decreasing, whence $\mu(a) \ge 1/\lambda_1 = 1/|\varphi'(\infty-)|$. And $f_3/f_2$ is non-decreasing too, so $\mu(x) \le 1/\lambda_3 = 1/|\varphi'(a+)|$.

As to part~(e), it suffices to prove the inequalities for $M(a,b)$ and $\mu(a,b)$ with $b < \infty$, because $t := \varphi'(a+)(b - a) \to -\infty$ and $t^2 N(t) \to 1$, $|t| \nu(t) \to 1$ as $b \to \infty$, see Proposition~\ref{prop:N.and.nu}. For $z \in [0,b-a]$, let $\psi_2(z) = \varphi(a+z) - \varphi(a)$ as before and
\[
	\psi_4(z) \ := \ \varphi'(a+) z - L z^2/2 .
\]
The difference $\psi_2 - \psi_4$ is non-negative and non-decreasing, because
\[
	\psi_2(z) - \psi_4(z) \ = \ \int_0^z \bigl[ \varphi'(a+t) - \varphi'(a+) + L t \bigr] \, \d t ,
\]
and the integrand is non-negative by Lipschitz-continuity of $\varphi'$ on $(a,b)$ with constant $L$. Hence, $M(a,b)$ is not smaller than
\begin{align*}
	f(a) \int_0^{b-a} z \exp(\psi_4(z)) \, \d z \
	&= \ f(a) \int_0^{b-a} z \exp(\varphi'(a+) z - L z^2/2) \, \d z \\
	&= \ f(a) (b - a)^2 \int_0^1 u \exp \bigl( \varphi'(a+)(b - a) u - c u^2) \, \d u \\
	&= \ f(a) (b - a)^2 N \bigl( \varphi'(a+)(b-a) \bigr) \Ex \exp(- c U^2) ,
\end{align*}
where $c := L (b - a)^2/2$, and $U$ denotes a nonnegative random variable with density proportional to $1_{[u < 1]} u \exp( \varphi'(a+)(b - a) u)$. Similarly, $\mu(a,b)$ is not smaller than
\begin{align*}
	\int_0^{b-a} & z \exp(\psi_4(z)) \, \d z \Big/ \int_0^{b-a} \exp(\psi_4(z)) \, \d z \\
	&\ge \ \int_0^{b-a} z \exp(\varphi'(a+) z - L z^2/2) \, \d z
		\Big/ \int_0^{b-a} \exp(\varphi'(a+) z) \, \d z \\
	&= \ (b - a) \int_0^1 u \exp \bigl( \varphi'(a+)(b - a) u - c u^2) \, \d u
		\Big/ \int_0^{1} \exp(\varphi'(a+)(b-a) u) \, \d u \\
	&= \ (b - a) \nu \bigl( \varphi'(a+)(b - a) \bigr) \Ex \exp( - c U^2) .
\end{align*}
But the random variable $U$ is stochastically smaller than a gamma random variable $Y$ with shape parameter $2$ and rate parameter $|\varphi'(a+)|(b-a)$. Thus it follows from Jensen's inequality and this comparison that
\[
	\Ex \exp(-c U^2) \ \ge \ \exp(- c \Ex(U^2)) \
	\ge \ \exp(- c \Ex(Y^2)) \
	= \ \exp \bigl( - 3 L \varphi'(a+)^{-2} \bigr) ,
\]
because $\Ex(Y^2) \ = \ 6 \varphi'(a+)^{-2} (b - a)^{-2}$.
\end{proof}

\subsection{Exponential and maximal inequalities}
\label{subsec:exp}

In what follows, let $\Mnhat, \munhat, \Wnhat$ and $\Mnhatemp, \munhatemp, \Wnhatemp$ be defined as $M, \mu, W$ with $\Pnhat$ and $\Pnhatemp$, respectively, in place of $P$. The following basic result will be our key to bound the ratios $\Mnhatemp/M$ and $\munhatemp/\mu$.

\begin{Proposition}
\label{prop:basic.exp}
Let $Q$ be a distribution on $[0,\infty)$ with log-concave density and mean $\mu_Q > 0$. Consider an i.i.d.\ sample $Y_1,\ldots,Y_m$ drawn from $Q$. Then, for arbitrary $t < 1$,
\[
	\Ex \exp \Bigl( t \sum_{i=1}^m \frac{Y_i}{\mu_Q} \Bigr) \ \le \ (1 - t)_{}^{-m} .
\]
\end{Proposition}

\begin{proof}[Proof of Proposition~\ref{prop:basic.exp}]
If $Q$ follows an exponential distribution, it is well known that the above inequality holds with equality. Now suppose that $Q$ is an arbitrary distribution with log-density $\varphi$. Let $\tilde{\varphi}$ be the log-density of the exponential distribution $\tilde{Q}$ with mean $\mu_Q$, that is, $\tilde{\varphi}(x) = - x/\mu_Q - \log(\mu_Q)$ for $x \ge 0$ and $\tilde{\varphi}(x) = -\infty$ for $x < 0$. Concavity of $\varphi$ and linearity of $\tilde{\varphi}$ on $[0,\infty)$ together with equality of the means imply that for suitable real numbers $0 < a < b$, $\varphi(x) \le \tilde{\varphi}(x)$ for $x \not\in [a,b]$ and $\varphi(x) \ge \tilde{\varphi}(x)$ for $x \in (a,b)$. By Lemma~b in~\cite{Karlin_1963} (see also Theorem 3.A.44 in~\cite{Shaked_Shanthikumar_2007}), we obtain
\begin{equation}
\label{ineq:convex.order}
	\int \Psi \, \d Q \ \le \ \int \Psi \, \d \tilde{Q}
	\quad\text{for all convex} \ \Psi : \R \to \R .
\end{equation}
Applying \eqref{ineq:convex.order} to $\Psi(x) = \exp(tx/\mu_Q)$ for arbitrary $t < 1$, and using independence of $Y_1, \ldots, Y_m$ concludes the proof.
\end{proof}

Our next key results are simultaneous inequalities for the univariate versions of $\munhatemp/\mu$ and $\munhat/\mu$.

\begin{Proposition}
\label{prop:munhat(emp)}
{\bf (a)} \ For any $\tau > 1$, the probability that
\[
	\Bigl| \frac{\munhatemp(X_{(k)})}{\mu(X_{(k)})} - 1 \Bigr| \
	< \ \sqrt{ \frac{2\tau \log(n)}{n - k} } + \frac{\tau \log(n)}{n-k} \ \
	\text{for} \ k = 1,\ldots,n-1
\]
is at least $1 - 2 n^{1 - \tau}$.

\noindent
{\bf (b)} \ For arbitrary constants $b_n$ such that $b_n \to b_o$ and $(1 - F(b_n)) / \rho_n \to \infty$,
\[
	\max_{k < n : X_{(k)} \le b_n} \,
		\Bigl| \frac{\munhatemp(X_{(k)})}{\mu(X_{(k)})} - 1 \Bigr| \
	= \ \Op \bigl( \sqrt{\rho_n/(1 - F(b_n)} \bigr) ,
\]
\[
	\max_{x \in \Snhat : x \le b_n} \,
		\Bigl| \frac{\munhat(x)}{\mu(x)} - 1 \Bigr| \
	= \ \Op \bigl( \sqrt{\rho_n/(1 - F(b_n)} \bigr) .
\]
\end{Proposition}

\begin{proof}[Proof of Proposition~\ref{prop:munhat(emp)}]
As to part~(a), suppose first that $P$ is the exponential distribution with mean $\mu(0)$. Then Chernov's bound, applied to exponential distributions, shows that for $\eps \ge 0$,
\begin{equation}
\label{ineq:Chernov.exp}
	\Pr \biggl( \pm \Bigl( \frac{\munhatemp(0)}{\mu(0)} - 1 \Bigr) \ge \eps \biggr) \
	\le \ \exp[ - n H(\pm \eps)] ,
\end{equation}
where $H(t) := t - \log(1 + t)$ for $t > -1$ and $H(t) := \infty$ for $t \le -1$. 
Due to Proposition~\ref{prop:basic.exp} the Chernov bound \eqref{ineq:Chernov.exp} holds true for arbitrary distributions $P$ with log-concave density such that $a_o \ge 0$.
Coming back to the general case, note that for any $k \in \{1,\ldots,n-1\}$ and $a < b_o$, the conditional distribution of $(X_{(k + \ell)} - a)_{\ell = 1}^{n-k}$, given that $X_{(k)} = a$, coincides with the distribution of $(Y_{(\ell)})_{\ell=1}^{n-k}$, where $Y_{(1)} \le \cdots \le Y_{(n-k)}$ are the order statistics of $n -k$ independent random variables with density $f_a(y) := 1_{[y \ge 0]} f(a + y) /(1 - F(a))$. Since $\munhatemp(a)$ is the mean of $X_{(k + \ell)} - a$, $1 \le \ell \le n-k$, we may apply the inequalities \eqref{ineq:Chernov.exp} to deduce that
\begin{equation}
\label{ineq:Chernov.Xk}
	\Pr \biggl( \pm \Bigl( \frac{\munhatemp(X_{(k)})}{\mu(X_{(k)})} - 1 \Bigr) \ge \eps \biggr) \
	\le \ \exp[ - (n-k) H(\pm \eps)]
\end{equation}
for arbitrary $\eps \ge 0$. Since $H(-\eps) \ge H(\eps)$ for all $\eps \ge 0$, this implies that
\begin{equation}
\label{ineq:Chernov.Xk'}
	\Pr \biggl( \Bigl| \frac{\munhatemp(X_{(k)})}{\mu(X_{(k)})} - 1 \Bigr| \ge \eps \biggr) \
	\le \ 2 \exp[ - (n-k) H(\eps)] .
\end{equation}
Note that $H \bigl( \sqrt{2r} + r \bigr) \ge r$ for arbitrary $r \ge 0$, because $H \bigl( \sqrt{2r} + r \bigr) - r$ is equal to $\sqrt{2r} - \log \bigl( 1 + \sqrt{2r} + r \bigr)$, and $\exp \bigl( \sqrt{2r} \bigr) \ge 1 + \sqrt{2r} + \sqrt{2r}^2/2 = 1 + \sqrt{2r} + r$. Consequently,
\begin{align*}
	\Pr & \biggl( \Bigl| \frac{\munhatemp(X_{(k)})}{\mu(X_{(k)})} - 1 \Bigr|
		\ge \sqrt{ \frac{2\tau \log(n)}{n - k} } + \frac{\tau \log(n)}{n-k} \ \
		\text{for some} \ k \in \{1,\ldots,n-1\} \biggr) \\
	&\le \ \sum_{k=1}^{n-1}
		\Pr \biggl( \Bigl| \frac{\munhatemp(X_{(k)})}{\mu(X_{(k)})} - 1 \Bigr|
		\ge \sqrt{ \frac{2\tau \log(n)}{n - k} } + \frac{\tau \log(n)}{n-k} \biggr) \\
	&\le \ 2 \sum_{k=1}^{n-1} \exp( - \tau \log(n)) \
		< \ 2 n^{1 - \tau} .
\end{align*}

Concerning part~(b), note that $\log(n)/(n - k)$ equals $\rho_n / (1 - \Fnhatemp(X_{(k)})$, so combining part~(a) with \eqref{eq:tails.Fnhatemp} shows that the maximum of $\bigl| \munhatemp(X_{(k)})/\mu(X_{(k)}) - 1 \bigr|$ over all $k$ such that $X_{(k)} \le b_n$ is of order $\Op \bigl( \sqrt{\rho_n/(1 - F(b_n)} \bigr)$. Combining part~(a) with \eqref{eq:char.1}, \eqref{eq:tails.Fnhatemp} and \eqref{eq:tails.Fnhat} shows that the maximum of $\bigl| \munhat(x)/\mu(x) - 1 \bigr|$ over all $x \in \Snhat$ such that $x \le b_n$ is of order $\Op \bigl( \sqrt{\rho_n/(1 - F(b_n)} \bigr)$ too.
\end{proof}

Finally, we need some inequalities for the ratios $\Mnhatemp(a,b)/M(a,b)$, $\Mnhat(a,b)/M(a,b)$ over a broad range of pairs $(a,b)$.

\begin{Proposition}
\label{prop:Mnhat(emp)}
\textbf{(a)} \ For any real number $a$ and $b \in (a,\infty]$ such that $P(a,b) > 0$,
\[
	\Pr \biggl( \Bigl| \frac{\Mnhatemp(a,b)}{M(a,b)} - 1 \Bigr|
		\ge 2 \sqrt{\frac{\tau}{P(a,b)}} + \frac{\tau}{P(a,b)} \biggr) \
	\le \ 2 e_{}^{-n\tau}
\]
for all $\tau > 0$.

\noindent
\textbf{(b)} \ For any sequence of numbers $\delta_n \in (0,1)$ such that $\delta_n \to 0$ and $\delta_n/\rho_n \to \infty$, let $\AA_n := \bigl\{ (a,b) : -\infty < a < b \le \infty, P(a,b) \ge \delta_n \bigr\}$. Then
\[
	\sup_{(a,b) \in \AA_n} \,
		\Bigl| \frac{\Mnhatemp(a,b)}{M(a,b)} - 1 \Bigr| \ \top \ 0 .
\]
Moreover,
\[
	\sup_{(a,b) \in \AA_n : b \in \Snhat} \,
		\Bigl( \frac{\Mnhat(a,b)}{M(a,b)} - 1 \Bigr)^+ \ \top \ 0 .
\]
\end{Proposition}

\begin{proof}[Proof of Proposition~\ref{prop:Mnhat(emp)}]
As to part~(a), we first note that Proposition~\ref{prop:basic.exp} yields the inequality
\begin{align*}
	\Ex \exp \Bigl( t \frac{\Mnhatemp(a,b)}{M(a,b)} \Bigr) \
	&= \ \Ex \exp \Bigl( \frac{t}{n P(a,b)} \sum_{i=1}^n 1_{(a,b)}(X_i) \frac{X_i - a}{\mu(a,b)} \Bigr) \\
	&= \ \Ex \Ex \Bigl( \exp \Bigl( \frac{t}{n P(a,b)}
		\sum_{i=1}^n 1_{(a,b)}(X_i) \frac{X_i - a}{\mu(a,b)} \Bigr) \, \Big|\, \Pnhatemp(a,b) \Bigr) \\
	&\le \ \sum_{k=0}^\infty \binom{n}{k} P(a,b)^k (1 - P(a,b))^{n-k}
		\Bigl( 1 - \frac{t}{nP(a,b)} \Bigr)^{-k} \\
	&= \ \biggl( 1 + \frac{P(a,b) t}{nP(a,b) - t} \biggr)^n \\
	&\le \ \exp \Bigl( \frac{n P(a,b) t}{nP(a,b) - t} \Bigr) \
		= \ \exp \Bigl( t + \frac{t^2}{nP(a,b) - t} \Bigr)
\end{align*}
for arbitrary $t < nP(a,b)$, so
\[
	\Ex \exp \Bigl( t \Bigl( \frac{\Mnhatemp(a,b)}{M(a,b)} - 1 \Bigr) \Bigr) \
	\le \ \exp \Bigl( \frac{t^2}{nP(a,b) - t} \Bigr) .
\]
A standard application of Markov's inequality shows that for $\eta \ge 0$,
\[
	\Pr \biggl( \Bigl| \frac{\Mnhatemp(a,b)}{M(a,b)} - 1 \Bigr| \ge \eta \biggr) \
	\le \ 2 \exp \Bigl( \frac{t^2}{nP(a,b) - t} - t \eta \Bigr)
\]
for all $t \in [0, n P(a,b))$. This bound is minimized for $t = n P(a,b) \bigl( 1 - 1/\sqrt{1+ \eta} \bigr)$, which leads to
\[
	\Pr\biggl( \Bigl| \frac{\Mnhatemp(a,b)}{M(a,b)} - 1 \Bigr| \ge \eta \biggr) \
	\le \ 2\exp \Bigl( - n P(a,b) \bigl( \sqrt{1+\eta} - 1\bigr)^2 \Bigr) .
\]
Thus, for $\eta = 2 \sqrt{\tau / P(a,b)} + \tau/P(a,b)$ we obtain the asserted inequality.

Concerning part~(b), let $\DD_n = \bigl\{ j/n : j \in \mathbb{Z} \cap [-n^2, n^2] \bigr\} \cup \{\infty\}$. Then it follows from part~(a) that for $D > 4$,
\begin{align*}
	\Pr \biggl( & \Bigl| \frac{\Mnhatemp(a,b)}{M(a,b)} - 1 \Bigr| \ge
			2 \sqrt{ \frac{D \rho_n}{P(a,b)} } + \frac{D \rho_n}{P(a,b)} 
			\ \text{for some} \ a,b \in \DD_n, P(a,b) > 0 \biggr) \\
	&\le \ (2 n^2 + 1) 2 n^2 n^{-D} \ \to \ 0 .
\end{align*}
In particular,
\[
	\Delta_{n,M} \ := \ \max_{a,b \in \DD_n : P(a,b) \ge \delta_n/2} \,
		\Bigl| \frac{\Mnhatemp(a,b)}{M(a,b)} - 1 \Bigr| \
	\top \ 0 .
\]
Note also that by Proposition~\ref{prop:Pnhatemp},
\[
	\Delta_{n,P} \ := \ \max_{a,b \in \DD_n : P(a,b) \ge \delta_n/2} \,
		\Bigl| \frac{\Pnhatemp(a,b)}{P(a,b)} - 1 \Bigr| \
	\top \ 0 .
\]

For an arbitrary pair $(a,b) \in \AA_n$, let $(a',a'']$ and $(b',b'']$ be minimal intervals with endpoints in $\{-\infty\} \cup \DD_n$ containing $a$ and $b$, respectively. Note that $P(a',a'')$ and $P(b',b'')$ are not larger than
\[
	\gamma_n \ := \ \max \bigl( \{ P(c,c + 1/n) : c \in \R\} \cup \{P(-\infty,-n), P(n,\infty)\} \bigr) \
	= \ O(n^{-1}) ,
\]
because $P$ has a bounded density and subexponential tails. In particular, $P(a'',b') \ge \delta_n - 2 \gamma_n \ge \delta_n/2$ for sufficiently large $n$. On the one hand,
\[
	\frac{\Mnhatemp(a,b)}{M(a,b)} \
	\le \ \frac{\Mnhatemp(a,b'')}{M(a,b'')} \cdot \frac{M(a,b'')}{M(a,b)}
\]
and
\[
	\frac{\Mnhatemp(a,b)}{M(a,b)} \
	\ge \ \frac{\Mnhatemp(a,b')}{M(a,b')} \cdot \frac{M(a,b')}{M(a,b)} ,
\]
and by Proposition~\ref{prop:M.and.mu}~(b), the ratios $M(a,b'')/M(a,b)$ and $M(a,b)/M(a,b')$ are larger than one but not larger than $[p_n + (1 - p_n) \log(1 - p_n)]^{-1} \to 1$, where $p_n := (\delta_n - \gamma_n)^+/\delta_n \to 1$. On the other hand, it follows from Fubini's theorem that uniformly in $(a,b) \in \AA_n$ and for sufficiently large $n$,
\begin{align*}
	\frac{\Mnhatemp(a,b'')}{M(a,b'')} \
	&= \ \frac{\int_a^{a''} \Pnhatemp(x,b'') \, \d x + \Mnhatemp(a'',b'')}
		{\int_a^{a''} P(x,b'') \, \d x + M(a'',b'')} \\
	&\le \ \frac{(1 + \Delta_{n,P}) \int_a^{a''} P(x,b'') \, \d x + (1 + \Delta_{n,M}) M(a'',b'')}
		{\int_a^{a''} P(x,b'') \, \d x + M(a'',b'')} \\
	&\le \ 1 + \max\{\Delta_{n,P},\Delta_{n,M}\}
\end{align*}
and
\begin{align*}
	\frac{\Mnhatemp(a,b')}{M(a,b')} \
	&= \ \frac{\int_a^{a''} \Pnhatemp(x,b') \, \d x + \Mnhatemp(a'',b')}
		{\int_a^{a''} P(x,b'') \, \d x + M(a'',b'')} \\
	&\ge \ \frac{(1 - \Delta_{n,P}) \int_a^{a''} P(x,b') \, \d x + (1 - \Delta_{n,M}) M(a'',b')}
		{\int_a^{a''} P(x,b') \, \d x + M(a'',b')} \\
	&\ge \ 1 - \max\{\Delta_{n,P},\Delta_{n,M}\} .
\end{align*}
These considerations show that the supremum of $\bigl| \Mnhatemp(a,b) / M(a,b) - 1 \bigr|$ over all $(a,b) \in \AA_n$ converges to $0$ in probability.

Concerning the ratio $\Mnhat(a,b)/M(a,b)$, note first that for any probability distribution $Q$ with finite first moment and real numbers $a < b$,
\[
	\int (x - a)^+ \, Q(\d x) - \int (x - b)^+ \, Q(\d x) \ 
	= \ \int_{(a,b]} (x - a) Q(\d x) + (b - a) Q(b,\infty) .
\]
Consequently, it follows from \eqref{eq:char.1} and \eqref{eq:char.2} that
\[
	\lim_{b' \to b+} \Mnhatemp(a,b') \
	\ge \ \Mnhat(a,b) \quad\text{if} \ b \in \Snhat .
\]
This inequality implies the assertion about the ratio $\Mnhat(a,b)/M(a,b)$. 
\end{proof}

By symmetry the following results are an immediate consequence of Proposition~\ref{prop:Mnhat(emp)}.

\begin{Corollary}
\label{cor:Wnhat(emp)}
For any sequence of numbers $\delta_n \in (0,1)$ such that $\delta_n \to 0$ and $\delta_n/\rho_n \to \infty$, let $\tilde{\AA}_n := \bigl\{ (a,b) : -\infty \le a < b < \infty, P(a,b) \ge \delta_n \bigr\}$. Then
\[
	\sup_{(a,b) \in \tilde{\AA}_n} \,
		\Bigl| \frac{\Wnhatemp(a,b)}{W(a,b)} - 1 \Bigr| \ \top \ 0 .
\]
Moreover,
\[
	\sup_{(a,b) \in \tilde{\AA}_n : a \in \Snhat} \,
		\Bigl( \frac{\Wnhat(a,b)}{W(a,b)} - 1 \Bigr)^+ \ \top \ 0 .
\]
\end{Corollary}

\section{Proofs of the main results}
\label{sec:Proofs}

\begin{proof}[Proof of Theorem~\ref{thm0}]
Since $f$ and $\fnhat$ are zero on $\R \setminus (a_o,b_o)$, and because of \eqref{eq:consistency.f.ab}, it suffices to show that for fixed points $a_o < a < b < b_o$,
\[
	\sup_{x \in (a_o,a]} \, \bigl( \fnhat(x) - f(x) \bigr) \ \le \ L(a) + \op(1)
	\quad\text{and}\quad
	\sup_{x \in [b,b_o)} \, \bigl( \fnhat(x) - f(x) \bigr) \ \le \ R(b) + \op(1)
\]
with bounds $L(a), R(b)$ such that $L(a) \to 0$ as $a \downarrow a_o$ and $R(b) \to 0$ as $b \uparrow b_o$. For symmetry reasons we only consider the second claim. We fix an arbitrary $m \in (a_o,b_o)$ such that $\varphi'(m+) < 0$ in case of $\varphi'(b_o-) < 0$ and restrict our attention to $b \in (m,b_o)$. If $b_o < \infty$ and $\varphi'(b_o-) \ge 0$, then concavity of $\varphi$ and $\phinhat$ implies that
\begin{align*}
	\sup_{x \in [b,b_o)} \, \bigl( \fnhat(x) - f(x) \bigr) \
	&\le \ \sup_{x \in [b,b_o)} \, \fnhat(b) \exp \Bigl( \frac{\phinhat(b) - \phinhat(m)}{b-m} (x - b) \Bigr)
		- f(b) \\
	&\top \ f(b) \bigl( f(b)/f(m) \bigr)^{(b_o - b)/(b - m)} - f(b) \\
	&\le \ f(b_o) \bigl[ \bigl( f(b_o)/f(m) \bigr)^{(b_o - b)/(b - m)} - 1 \bigr] \ =: \ R(b) ,
\end{align*}
and $R(b) \to 0$ as $b \uparrow b_o$. If $\varphi'(b_o-) < 0$, then
\begin{align*}
	\sup_{x \in [b,b_o)} \, \bigl( \fnhat(x) - f(x) \bigr) \
	&\le \ \sup_{x \in [b,b_o)} \, \fnhat(b) \exp \Bigl( \frac{\phinhat(b) - \phinhat(m)}{b-m} (x - b) \Bigr)
		- f(b_o) \\
	&\top \ f(b) - f(b_o) \ =: \ R(b) ,
\end{align*}
because $\bigl( \phinhat(b) - \phinhat(m) \bigr) / (b - m) \top \bigl( \varphi(b) - \varphi(m) \bigr) / (b - m) \le \varphi'(m+) < 0$, and $R(b) \to 0$ as $b \uparrow b_o$. Here $f(\infty) := 0$.

Let $(b_n)_n$ be a sequence in $(a_o,b_o)$ with limit $b_o$. For arbitrary fixed $a_o < a < b < b_o$, it follows from concavity of $\phinhat$ that for sufficiently large $n$,
\[
	\phinhat'(b_n+) \ \le \ \frac{\phinhat(b) - \phinhat(a)}{b - a} \
	\top \ \frac{\varphi(b) - \varphi(a)}{b - a} .
\]
Now the assertion follows from the fact that the right-hand side converges to $\varphi'(b_o-)$ as $a,b \to b_o$.
\end{proof}

\begin{proof}[Proof of Theorem~\ref{thm:finite.b0}]
As to part~(a), it follows from \eqref{eq:consistency.f.ab} and Theorem~\ref{thm0} that for any fixed $b \in (a,b_o)$,
\begin{align*}
	\sup_{x \ge a} \, \bigl| \fnhat(x) - f(x) \bigr| \
	&\le \ \sup_{x \in [a,b]} \, \bigl| \fnhat(x) - f(x) \bigr|
		+ \sup_{x \ge b} \, \bigl| \fnhat(x) - f(x) \bigr| \\
	&\le \ \op(1) + \sup_{x \ge b} \, \bigl( \fnhat(x) - f(x) \bigr)^+
		+ \sup_{x \ge b} \, f(x) \\
	&= \ \op(1) + \sup_{x \ge b} \, f(x) ,
\end{align*}
and $\sup_{x \ge b} f(x) \to f(b_o) = 0$ as $b \uparrow b_o$. Furthermore, since $\varphi'(b_o-) = -\infty$, it follows from Theorem~\ref{thm0} that $\phinhat'(b_n+) \top \varphi'(b_o-)$.

As to part~(b), it follows from \eqref{eq:consistency.phi.ab} and Theorem~\ref{thm0} that for any fixed $b \in (a,b_o)$ and all $n$ with $b_n > b$,
\begin{align*}
	\sup_{x \in [a,b_n]} \, \bigl| \phinhat(x) - \varphi(x) \bigr| \
	&\le \ \sup_{x \in [a,b]} \, \bigl| \phinhat(x) - \varphi(x) \bigr|
		+ \sup_{x \in [b,b_n]} \, \bigl| \phinhat(x) - \varphi(x) \bigr| \\
	&\le \ \op(1)
		+ \sup_{x \in [b,b_o]} \, \bigl( \phinhat(x) - \varphi(x) \bigr)^+
		+ \sup_{x \in [b,b_n]} \, \bigl( \varphi(x) - \phinhat(x) \bigr)^+ \\
	&= \ \op(1)
		+ \sup_{x \in [b,b_n]} \, \bigl( \varphi(x) - \phinhat(x) \bigr)^+ ,
\end{align*}
where we used the fact that $\delta(b) := \min_{x \in [b,b_o]} f(x) > 0$, so $(\phinhat - \varphi)^+ \le (\fnhat - f)^+/\delta(b)$ on $[b,b_o]$. By concavity of $\phinhat$,
\begin{align*}
	\sup_{x \in [b,b_n]} \bigl( \varphi(x) - \phinhat(x) \bigr)^+
	&\le \sup_{x \in [b,b_n]} \bigl( \varphi(x) - \varphi(b_o) \bigr)^+
		+ \sup_{x \in [b,b_n]} \bigl( \varphi(b_o) - \phinhat(x) \bigr)^+ \\
	&= \sup_{x \in [b,b_o]} \, \bigl( \varphi(x) - \varphi(b_o) \bigr)^+
		+ \max_{x \in \{b,b_n\}} \, \bigl( \varphi(b_o) - \phinhat(x) \bigr)^+ \\
	&\le 2 \sup_{x \in [b,b_o]} \bigl| \varphi(x) - \varphi(b_o) \bigr|
		+ \bigl( \varphi(b) - \phinhat(b) \bigr)^+
		+ \bigl( \varphi(b_o) - \phinhat(b_n) \bigr)^+ \\
	&= 2 \sup_{x \in [b,b_o]} \bigl| \varphi(x) - \varphi(b_o) \bigr|
		+ \op(1)
		+ \bigl( \varphi(b_o) - \phinhat(b_n) \bigr)^+
\end{align*}
by \eqref{eq:consistency.phi.ab}. Since $\sup_{x \in [b,b_o]} \bigl| \varphi(x) - \varphi(b_o) \bigr| \to 0$ as $b \uparrow b_o$, it suffices to show that
\[
	\bigl( \varphi(b_o) - \phinhat(b_n) \bigr)^+ \
	\top \ 0 .
\]
To this end, we show that for any fixed $\eps > 0$, the inequality $\phinhat(b_n) \le \varphi(b_o) - \eps$ holds with asymptotic probability zero. Let $b(\eps) \in (a_o,b_o)$ such that $|\varphi - \varphi(b_o)| \le \lambda \eps/2$ on $[b(\eps),b_o]$ for some $\lambda \in (0,1)$ to be specified later. From \eqref{eq:consistency.phi.ab} and Theorem~\ref{thm0} we may conclude that $\phinhat(b(\eps)) \ge \varphi(b_o) - \lambda\eps$ and $\phinhat \le \varphi(b_o) + \lambda \eps$ on $[b(\eps),b_o]$ with asymptotic probability one. Thus it suffices to show that the event
\[
	A_{n,\eps} \ := \ \bigl[ \phinhat(b_n) \le \varphi(b_o) - \eps,
		\phinhat(b(\eps)) \ge \varphi(b_o) - \lambda\eps,
		\phinhat \le \varphi(b_o) + \lambda\eps \ \text{on} \ [b(\eps),b_o] \bigr]
\]
has asymptotic probability zero. From now on we assume that the event $A_{n,\eps}$ occurs. Suppose that $n$ is sufficiently large such that $b_n > b(\eps)$. Note that $\phinhat'(b_n+) < 0$, because $\phinhat(b(\eps)) > \phinhat(b_n)$. Let $Y_n$ be the largest point in $\Snhat \cap (a_o,b_n]$. Then, $\phinhat$ is affine on $[Y_n,b_n]$ and non-increasing on $[Y_n,b_o]$. Consequently, $\fnhat$ is convex on $[Y_n,b_n]$ and non-increasing on $[Y_n,b_o]$. 

Suppose first that $Y_n \ge b(\eps)$. Then the properties of $\fnhat$ on $[Y_n,b_o]$ imply that
\begin{align*}
	1 - \Fnhat(Y_n) \ = \ \Pnhat([Y_n,b_o]) \
	&\le \ (b_n - Y_n) \frac{\fnhat(Y_n) + \fnhat(b_n)}{2} + (b_o - b_n) \fnhat(b_n) \\
	&\le \ (b_o - Y_n) \frac{\fnhat(Y_n) + \fnhat(b_n)}{2} \\
	&\le \ (b_o - Y_n) f(b_o) \frac{e^{\lambda\eps} + e^{-\eps}}{2} ,
\end{align*}
whereas
\[
	1 - F(Y_n) \ = \ P([Y_n,b_o]) \
	\ge \ (b_o - Y_n) f(b_o) e^{-\lambda\eps/2} .
\]
Consequently, for sufficiently large $n$, the event $A_{n,\eps}$ implies that
\[
	\frac{1 - \Fnhat(Y_n)}{1 - F(Y_n)} \
	\le \ \frac{e^{3\lambda\eps/2} + e^{\lambda\eps/2 - \eps}}{2} \
	= \ 1 + (\lambda - 1/2) \eps + O(\eps^2)
\]
as $\eps \downarrow 0$. Hence, if $\lambda < 1/2$ and $\eps > 0$ is sufficiently small, it follows from \eqref{eq:tails.Fnhat} that the event $A_{n,\eps}$ has asymptotic probability zero.

Suppose that $Y_n \le b(\eps)$. Then,
\begin{align*}
	\Pnhat([b(\eps),b_o]) \
	&\le \ (b_n - b(\eps)) \frac{\fnhat(b(\eps)) + \fnhat(b_n)}{2} + (b_o - b_n) \fnhat(b_n) \\
	&\le \ (b_o - b(\eps)) f(b_o) \frac{e^{\lambda\eps} + e^{-\eps}}{2} ,
\end{align*}
whereas
\[
	P([Y_n,b_o]) \
	\ge \ (b_o - b(\eps)) f(b_o) e^{-\lambda\eps/2} .
\]
Consequently, for sufficiently large $n$, the event $A_{n,\eps}$ implies that
\begin{align*}
	(P - \Pnhat)([b(\eps),b_o]) \
	&\ge \ (b_o - b(\eps)) f(b_o)
		\Bigl( e^{-\lambda\eps/2} - \frac{e^{\lambda\eps} + e^{-\eps}}{2} \Bigr) \\
	&= \ (b_o - b(\eps)) f(b_o)
		\bigl( (1/2 - \lambda)\eps + O(\eps^2) \bigr)
\end{align*}
as $\eps \downarrow 0$. Hence, if $\lambda < 1/2$ and $\eps > 0$ is sufficiently small, it follows from \eqref{eq:consistencyTV} that the event $A_{n,\eps}$ has asymptotic probability zero.

As to part~(c), because of Theorem~\ref{thm0}, it suffices to show that for any fixed $\eps > 0$, the event $B_{n,\eps} := \bigl[ \phinhat'(b_n+) < \varphi'(b_o-) - \eps \bigr]$ has asymptotic probability zero. Let $Y_n$ be the largest point in $\Snhat$ such that $Y_n \le b_n$. It follows from \eqref{eq:consistency.phi.ab} that for any fixed $a\in (a_o,b_o)$, the event $B_{n,\eps} \cap [Y_n \le a]$ has asymptotic probability zero, because for any fixed $b \in (a,b_o)$,
\[
	\phinhat'(a+) \ \ge \ \frac{\phinhat(b) - \phinhat(a)}{b - a} \
	\top \ \frac{\varphi(b) - \varphi(a)}{b - a} \ \ge \ \varphi'(b_o-) ,
\]
whereas $\phinhat'(Y_n+) = \phinhat'(b_n+)$. Consequently, there exist numbers $a_{n,\eps} \in (a_o,b_n)$ such that $a_{n,\eps} \to b_o$ and $\Pr(B_{n,\eps} \cap [Y_n \le a_{n,\eps}]) \to 0$. It remains to be shown that $B_{n,\eps} \cap [Y_n >a_{n,\eps}]$ has asymptotic probability zero. 
Assuming that the latter event occurs, note that by Proposition~\ref{prop:M.and.mu}~(c),
\begin{align*}
	\mu(Y_n) \
		&\ge \ (b_o - Y_n) \nu \bigl( \varphi'(b_o-)(b_o - Y_n) \bigr) , \\
	\munhat(Y_n) \
		&\le \ (b_o - Y_n) \nu \bigl( (\varphi'(b_o-) - \eps)(b_o - Y_n) \bigr) ,
\end{align*}
and since the moduli of $\varphi'(b_o-) (b_o - Y_n)$ and $(\varphi'(b_o-) - \eps)(b_o - Y_n)$ are not larger than $\bigl( |\varphi'(b_o-)| + \eps \bigr) (b_o - a_{n,\eps}) \to 0$, we may conclude that
\begin{align*}
	\frac{\munhat(Y_n)}{\mu(Y_n)} - 1 \
	&\le \ \frac{1/2 + (\varphi'(b_o-) - \eps)(b_o - Y_n)/12 + O(1) (b_o - Y_n)^2}
		{1/2 + \varphi'(b_o-) (b_o - Y_n)/12 + O(1) (b_o - Y_n)^2} - 1 \\
	&= \ - \frac{\eps(b_o - Y_n)/6 + O(1)(b_o - Y_n)^2}
		{1 + \varphi'(b_o-) (b_o - Y_n)/6 + O(1) (b_o - Y_n)^2} \\
	&= \ - (\eps + o(1)) (b_o - Y_n)/6 \\
	&\le \ - (\eps + o(1)) (b_o - b_n)/6 \\
	&= \ - \bigl( \eps/f(b_o) + o(1) \bigr) (1 - F(b_n)) ,
\end{align*}
uniformly on $B_{n,\eps} \cap [Y_n > a_{n,\eps}]$. On the other hand, we know from Proposition~\ref{prop:munhat(emp)}~(b) that
\[
	\frac{\munhat(Y_n)}{\mu(Y_n)} - 1 \
	\ge \ \Op \bigl( \sqrt{\rho_n / (1 - F(b_n))} \bigr) \
	= \ \op(1 - F(b_n)) ,
\]
because $\rho_n/(1 - F(b_n))^3 \to 0$ by assumption, whence $\Pr(B_{n,\eps} \cap [Y > a_{n,\eps}]) \to 0$.
\end{proof}

\begin{proof}[Proof of Theorem~\ref{thm:infinite.b0}]
Concerning part~(a), since $\{\fnhat > 0\} = [X_{(1)},X_{(n)}]$,
\[
	\Pr(\fnhat(b_n) = 0) \
	= \ \Pr(X_{(1)} > b_n) + \Pr(X_{(n)} < b_n) \
	= \ (1 - F(b_n))^n + F(b_n)^n \ \to \ 0 ,
\]
because $F(b_n) \to 1$ and $n(1 - F(b_n)) \to \infty$. 
If $\varphi'(\infty-) = -\infty$, it follows already from Theorem~\ref{thm0} that $\phinhat'(b_n+) \top \varphi'(\infty-)$. Otherwise, we know that $\phinhat'(b_n+) \le \varphi(\infty-) + \op(1)$, and it suffices to show that for any fixed $\eps > 0$, the inequality $\phinhat'(b_n+) \le \varphi'(\infty-) - \eps$ holds true with asymptotic probability zero. 
If $ \phinhat'(b_n+) \le \varphi'(\infty-) - \eps$, then it follows from Proposition~\ref{prop:M.and.mu}~(d) that $Y_n := \max(\Snhat \cap (a_o,b_n])$, satisfies
\begin{align*}
	\mu(Y_n) \
		&\ge \ -1/\varphi'(\infty-) , \\
	\munhat(Y_n) \
		&\le \ -1 / \phinhat'(Y_n+) \ = \ -1 / \phinhat'(b_n+) \ \le \ -1/(\varphi'(\infty-) - \eps) ,
\end{align*}
whence
\[
	\frac{\munhat(Y_n)}{\mu(Y_n)} \
	\le \ \frac{\varphi'(\infty-)}{\varphi'(\infty-) - \eps} \
	= \ \bigl( 1 + \eps / |\varphi'(\infty-)| \bigr)^{-1} .
\]
According to Proposition~\ref{prop:munhat(emp)}~(b), the latter inequality holds true with asymptotic probability zero.

To prove the claim in part~(b), recall that for any compact interval $I \subset (a_o,\infty)$,
\[
	\Delta_n(I) \ := \ \max_{x \in I} \, \bigl| \phinhat(x) - \varphi(x) \bigr| \ \top \ 0 .
\]
Consequently, there exists a sequence $(a_n)_n$ of numbers $a_n \in [a,b_n]$ converging to $\infty$ such that even $\Delta_n[a,a_n] \top 0$. Hence, it suffices to show that
\[
	\max_{x \in [a_n,b_n]} \, \frac{ \bigl| \phinhat(x) - \varphi(x) \bigr|}{1 + |\varphi(x)|} \
	\top \ 0 .
\]
We may assume without loss of generality that $\varphi(a_n) \le 0$ and $\varphi'(a_n+) < 0$. Applying part~(a) to $(b_n)_n$ and to $(a_n)_n$ in place of $(b_n)_n$ implies that
\[
	\Gamma_n \ := \ \max_{t \in [a_n,b_n]} \, \bigl| \phinhat'(t+)/\varphi'(t+) - 1 \bigr| \
	\top \ 0 .
\]
Thus, uniformly in $x \in [a_n,b_n]$,
\begin{align*}
	\bigl| \phinhat(x) - \varphi(x) \bigr| \
	&\le \ \bigl| \phinhat(a_n) - \varphi(a_n) \bigr| +
		\int_{a_n}^x \bigl| \phinhat'(t+) - \varphi'(t+) \bigr| \, \d t \\
	&\le \ \op(1) + \Gamma_n \int_{a_n}^x |\varphi'(t+)| \, \d t \\
	&= \ \op(1) + \Gamma_n \bigl( \varphi(a_n) - \varphi(x) \bigr) \\
	&\le \ \op(1) + \Gamma_n |\varphi(x)| .
\end{align*}

Now we prove the claims in part~(c). If $a < b < \infty$ and and $0 < \delta < a - a_o$, then for $x \in [a,b]$,
\begin{align*}
	\phinhat'(x+) - \varphi'(x) \
	&\le \ \frac{\phinhat(x) - \phinhat(x-\delta)}{\delta} - \varphi'(x) \\
	&\le \ 2 \Delta_n[a-\delta,b] / \delta
		+ \frac{1}{\delta} \int_{x-\delta}^x \bigl( \varphi'(s) - \varphi'(x) \bigr) \, \d s \\
	&\le \ 2 \Delta_n[a-\delta,b] / \delta
		+ L \delta/2 ,
\end{align*}
where $L$ denotes the Lipschitz constant of $\varphi'$ on $(a_*,\infty)$. Analogously,
\[
	\phinhat'(x+) - \varphi'(x) \
	\ge \ - 2 \Delta_n[a,b+\delta] / \delta
		- L \delta/2 ,
\]
Hence, $\max_{x \in [a,b]} \bigl| \phinhat'(x+) - \varphi'(x) \bigr| \top 0$. Consequently, there exist sequences $(a_n)_n$ and $(\eps_n)_n$ of numbers $a_n \in [a,b_n]$ and $\eps_n > 0$ such that $a_n \to \infty$, $\eps_n \to 0$, and with asymptotic probability one,
\begin{equation}
\label{ineq:aan}
	\max_{x \in [a,a_n]} \, \bigl| \phinhat(x) - \varphi(x) \bigr| \ < \ \eps_n ,
	\quad
	\sup_{x \in [a,a_n]} \, \bigl| \phinhat'(x+) - \varphi'(x) \bigr| \ < \ \eps_n .
\end{equation}
Consequently, it suffices to verify the claims of part~(c) with $[a_n,b_n]$ in place of $[a,b_n]$.

At first we show that
\begin{equation}
\label{ineq:sup.deriv}
	\Gamma_n \ := \ \sup_{x \in [a_n,b_n]} \, \Bigl( \frac{\phinhat'(x+)}{\varphi'(x)} - 1 \Bigr)^+ \
	\top \ 0 .
\end{equation}
It follows from Proposition~\ref{prop:M.and.mu}~(e,d,a,d) that
\[
	\sup_{x \in [a_n,b_n]} \, \frac{\phinhat'(x+)}{\varphi'(x)} \
	= \ (1 + o(1)) \sup_{x \in [a_n,b_n]} \, \mu(x) |\phinhat'(x+)| \
	\le \ (1 + o(1)) \max_{y \in \Snhat : y \le b_n} \, \frac{\mu(y)}{\munhat(y)} ,
\]
because for any $x \in [a_n,b_n]$, the point $y := \max(\Snhat \cap (a_o,x])$ satisfies $y \le x$ and $\phinhat'(y+) = \phinhat'(x+) \le \phinhat'(a_n+) < 0$ with asymptotic probability one, whence
\[
	\mu(x) |\phinhat'(x+)| \ \le \ \mu(y) |\phinhat'(x+)| \
	= \ \mu(y) |\phinhat'(y+)| \ \le \ \mu(y)/\munhat(y) .
\]
Now \eqref{ineq:sup.deriv} follows from the fact that the maximum of $\mu/\munhat$ over $\Snhat \cap (a_o,b_n]$ equals $1 + \op(1)$, see Proposition~\ref{prop:munhat(emp)}~(b).

Secondly, with \eqref{ineq:sup.deriv} at hand, the claim about $(\phinhat - \varphi)^+/(1 + |\varphi|)$ can be verified easily. Uniformly in $x \in [a_n,b_n]$,
\begin{align*}
	\varphi(x) - \phinhat(x) \
	&= \ \varphi(a_n) - \phinhat(a_n) +
		\int_{a_n}^x \bigl( \varphi'(t) - \phinhat'(t+) \bigr) \, \d t \\
	&\le \ \eps_n + \Gamma_n \int_{a_n}^x |\varphi'(t)| \, \d t \\
	&= \ \eps_n + \Gamma_n \bigl( \varphi(a_n) - \varphi(x) \bigr) \\
	&\le \ \op(1) \bigl( 1 + |\varphi(x)| \bigr) .
\end{align*}

Finally, we derive an upper bound for $(\phinhat - \varphi)^+$ on $[a_n,b_n]$. To this end we augment the boundary $a_n$ by a number $a_n' \in [a,a_n]$ such that
\[
	\varphi'(a_n') \ = \ \min\{\varphi'(a), \varphi'(a_n) + 1\} .
\]
Since $\varphi'(a_n') \to -\infty$, the sequence $(a_n')_n$ converges to $\infty$ too. Moreover, assuming that \eqref{ineq:aan} holds true, as soon as $\varphi'(a_n') = \varphi'(a_n) + 1$ and $\eps_n < 1/2$,
\[
	\phinhat'(a_n'+) - \phinhat'(a_n+) \
	\ge \ 1 - 2 \eps_n \ > \ 0 ,
\]
so
\begin{equation}
\label{eq:anp.an}
	\Pr \bigl( \Snhat \cap [a_n', a_n] \ne \emptyset) \ \to \ 1 .
\end{equation}

For any $x \ge a$ let
\[
	y(x) \ := \ x + |\varphi'(x)|^{-1} .
\]
This defines an increasing function $y : [a,\infty) \to (a,\infty)$ such that
\begin{equation}
\label{eq:Pxyx}
	\frac{P(x,y(x))}{P(x,\infty)} \ \to \ 1 - e^{-1} \quad\text{as} \ x \to \infty .
\end{equation}
Indeed,
\begin{align*}
	\frac{P(x,y(x))}{P(x,\infty)} \
	&= \ \int_0^{y(x) - x} \exp \bigl( \varphi(x+s) - \log \bar{F}(x) \bigr) \, \d s \\
	&\ge \ |\varphi'(x)| \int_0^{y(x) - x} \exp(- |\varphi'(x)| s) \, \d s \\
	&= \ 1 - e^{-1} ,
\end{align*}
because the distribution with density $s \mapsto \exp \bigl( \varphi(x+s) - \log \bar{F}(x) \bigr)$ on $(0,\infty)$ is stochastically smaller than the exponential distribution with rate $|\varphi'(x)|$. On the other hand, since $-|\varphi'(x)|s \ge \varphi(x+s) - \varphi(x) \ge - |\varphi'(x)| s - L s^2/2$ for $s \ge 0$,
\begin{align*}
	\frac{P(x,y(x))}{P(x,\infty)} \
	&\le \ \int_0^{y(x) - x} \exp \bigl( - |\varphi'(x)|s \bigr) \, \d s
		\Big/ \int_0^\infty \exp \bigl( - |\varphi'(x)| s - L s^2/2 \bigr) \, \d s \\
	&= \ (1 - e^{-1}) \Big/ \int_0^\infty e^{-y} \exp(- L \varphi'(x)^{-2} y^2/2 \bigr) \, \d y \\
	&\le \ \exp(L \varphi'(x)^{-2}) (1 - e^{-1})
\end{align*}
by Jensen's inequality.

Now we define $b_n' := y(b_n)$ and $b_n'' := y(b_n')$. Then it follows from \eqref{eq:Pxyx} that
\[
	\frac{1 - F(b_n'')}{\rho_n} \ = \ (e^{-2} + o(1)) \frac{1 - F(b_n)}{\rho_n} \ \to \ \infty
\]
Hence, \eqref{ineq:sup.deriv} remains valid with $[a_n',b_n'']$ in place of $[a_n,b_n]$. Furthermore, \eqref{eq:Pxyx} implies that
\[
	\inf_{x \in [a_n',b_n']} \frac{P(x,y(x))}{\rho_n} \
	= \ (1 - e^{-1} + o(1)) \frac{1 - F(b_n')}{\rho_n} \
	\to \ \infty .
\]
Consequently, by Corollary~\ref{cor:Wnhat(emp)},
\[
	\max_{x \in [a_n',b_n'] \cap \Snhat} \, \frac{\Wnhat(x,y(x))}{W(x,y(x))} \
	\le \ 1 + \op(1) .
\]
On the other hand, Proposition~\ref{prop:M.and.mu}~(c) implies that for $x \in [a_n',b_n']$,
\begin{align*}
	\Wnhat(x,y(x)) \
	&\ge \ \fnhat(x) \varphi'(x)^{-2} V \bigl( - \phinhat'(y(x)+) /\varphi(x) \bigr) , \\
	W(x,y(x)) \
	&\le \ f(x) \varphi'(x)^{-2} V(-1) ,
\end{align*}
because $\varphi' < 0$ on $[a,\infty)$, and it follows from \eqref{ineq:sup.deriv} with $[a_n',b_n'']$ in place of $[a_n,b_n]$, Lipschitz-continuity of $\varphi'$ on $[a,\infty)$ with constant $L$ and Proposition~\ref{prop:N.and.nu} that
\begin{align*}
	\inf_{x \in [a_n',b_n']} \,
		\frac{V \bigl( - \phinhat'(y(x)+) / \varphi(x) \bigr)}{V(-1)} \
	&\ge \ \inf_{x \in [a_n',b_n']} \,
		\frac{V \bigl( - (1 + \op(1)) \varphi'(y(x))/\varphi'(x) \bigr)}{V(-1)} \\
	&\ge \ \frac{V \bigl( - (1 + \op(1)) (1 + L \varphi'(a_n')^{-2}) \bigr)}{V(-1)} \\
	&= \ 1 + \op(1) ,
\end{align*}
because $\varphi'(y(x))/\varphi'(x) \ge - 1 - L \varphi'(x)^{-2} \ge -1 - L \varphi'(a_n')^{-2} \to -1$. Consequently,
\[
	\max_{x \in [a_n',b_n'] \cap \Snhat} \, \frac{\fnhat(x)}{f(x)} \ \le \ 1 + \op(1) ,
\]
which is equivalent to
\[
	\max_{x \in [a_n',b_n'] \cap \Snhat} \, \bigl( \phinhat(x) - \varphi(x) \bigr)^+ \
	= \ \op(1) .
\]
Since $\phinhat - \varphi$ is convex between consecutive knots in $\Snhat$, and since $\Snhat \cap [a_n',a_n] \ne \emptyset$ with asymptotic probability one by \eqref{eq:anp.an}, this implies that
\[
	\max_{x \in [a_n,\hat{b}_n]} \, \bigl( \phinhat(x) - \varphi(x) \bigr)^+ \
	= \ \op(1) ,
\]
where $\hat{b}_n := \max(\Snhat \cap (-\infty, b_n'])$. This proves already our claim in case of $\hat{b}_n \ge b_n$. So it remains to show that
\[
	\frac{\fnhat(b_n)}{f(b_n)} \ \le \ 1 + \op(1)
	\quad\text{if} \ \Snhat \cap [b_n,b_n'] = \emptyset.
\]
With $\check{b}_n := \min(\Snhat \cap [b_n',\infty))$, the inequality $\hat{b}_n \le b_n$ implies that $\phinhat$ is linear on $[b_n,\check{b}_n]$, and $P(b_n,\check{b}_n)/\rho_n \ge P(b_n,b_n')/\rho_n \ge (1 - e^{-1} + o(1)) (1 - F(b_n))/\rho_n \to \infty$. Hence Proposition~\ref{prop:Mnhat(emp)}~(b) implies that
\[
	\frac{\Mnhat(b_n,\check{b}_n)}{M(b_n,\check{b}_n)} \ \le \ 1 + \op(1) .
\]
On the other hand, Proposition~\ref{prop:M.and.mu}~(c,e) implies that
\begin{align*}
	\Mnhat(b_n,\check{b}_n) \
	&\ge \ \fnhat(b_n) (\check{b}_n - b_n)^2 N \bigl(  \phinhat'(b_n+) (\check{b}_n - b_n) \bigr) , \\
	M(b_n,\check{b}_n) \
	&\le \ f(b_n) (\check{b}_n - b_n)^2 N \bigl(  \varphi'(b_n) (\check{b}_n - b_n) \bigr)  ,
\end{align*}
and it follows from \eqref{ineq:sup.deriv}, Lipschitz-continuity of $\varphi'$ on $[a,\infty)$ with constant $L$ and Proposition~\ref{prop:N.and.nu} that
\[
	\frac{N \bigl(  \phinhat'(b_n+) (\check{b}_n - b_n) \bigr)}
		{N \bigl(  \varphi'(b_n) (\check{b}_n - b_n) \bigr)} \
	\ge \ \frac{N \bigl(  (1 + \op(1)) \varphi'(b_n) (\check{b}_n - b_n) \bigr)}
		{N \bigl(  \varphi'(b_n) (\check{b}_n - b_n) \bigr)} \
	= \ 1 + \op(1) .
\]
The latter conclusion follows from the fact that for arbitrary sequences $(t_n)_n$ in $(-\infty,0]$, $N \bigl( (1 + o(1)) t_n \bigr) / N(t_n) = 1 + o(1)$. For a bounded sequence $(t_n)_n$ this follows from continuity of $N$. If $t_n \to -\infty$, this follows from the expansion $N \bigl( (1 + o(1)) t_n \bigr) = (1 + o(1)) t_n^{-2}$.
\end{proof}

\paragraph{Acknowledgement.}
Part of this work was supported by the Swiss National Science Foundation.


\end{document}